\documentclass[a4paper, reqno, 11pt]{amsart}

\usepackage[english]{babel}
\usepackage{amsmath}
\usepackage{amssymb}
\usepackage{enumerate}
\usepackage{ifthen}
\usepackage{bbm}
\usepackage{graphicx,subfigure}  
\usepackage{geometry}
\provideboolean{shownotes} 
\setboolean{shownotes}{true}
\usepackage{hyperref}
\usepackage{setspace}
\usepackage{soul} 

\newcommand{\hole}[1]{
\ifthenelse{\boolean{shownotes}}%
{\begin{center} \fbox{ \rule {.25cm}{0cm}
\rule[-.1cm]{0cm}{.4cm} \parbox{.85\textwidth}{\begin{center}
\texttt{#1}\end{center}} \rule {.25cm}{0cm}}\end{center}}
{}
}

\newtheorem{theorem}{Theorem}[section]
\newtheorem{proposition}[theorem]{Proposition}
\newtheorem{lemma}[theorem]{Lemma}

\newtheorem{definition}[theorem]{Definition}

\allowdisplaybreaks
\theoremstyle{remark}
\newtheorem{remark}[theorem]{Remark}
\newtheorem{example}[theorem]{Example}

\newcommand{\R}{\mathbb{R}}
\newcommand{\T}{\mathbb{T}^d}
\newcommand{\Tt}{\mathbb{T}^2}
\newcommand{\N}{\mathbb{N}}
\newcommand{\Z}{\mathbb{Z}}
\newcommand{\e}{\varepsilon}

\newcommand{\dive}{\mathop{\mathrm {div}}}
\newcommand{\curl}{\mathop{\mathrm {curl}}}
\newcommand{\weakto}{\rightharpoonup}
\newcommand{\weaktos}{\stackrel{*}{\rightharpoonup}}

\newcommand{\del}{\partial}
\newcommand{\eps}{\varepsilon}
\newcommand{\Tone}{\mathbb{T}^1}
\newcommand{\cE}{\mathcal{E}}

\numberwithin{equation}{section}

\begin{document}

\title[\empty]{Existence and uniqueness for a viscoelastic Kelvin-Voigt model with nonconvex stored energy}

\author[K. Koumatos]{Konstantinos Koumatos}
\address[Konstantinos Koumatos]{\newline
Department of Mathematics
\newline
University of Sussex
\newline
Pevensey 2 Building
\newline
Falmer, Brighton, BN1 9QH, UK
}
\email{k.koumatos@sussex.ac.uk}
\author[C. Lattanzio]{Corrado Lattanzio}
\address[Corrado Lattanzio]{\newline 
Dipartimento di Ingegneria e Scienze dell'Informazione e Matematica
\newline
Universit\`a degli Studi dell'Aquila
\newline
Via Vetoio
\newline
I-67010 Coppito (L'Aquila) AQ 
\newline 
Italy
}
\email{corrado@univaq.it}
\author[S. Spirito]{Stefano Spirito}
\address[Stefano Spirito]{\newline 
Department of Information Engineering, Computer Science and Mathematics
\newline
University of L'Aquila
\newline
Via Vetoio
\newline
I-67100 Coppito (L'Aquila), 
Italy
}
\email{stefano.spirito@univaq.it}

\author[A.E. Tzavaras]{Athanasios E. Tzavaras}
\address[Athanasios E. Tzavaras]{
\newline
Computer, Electrical and Mathematical Science and Engineering Division 
\newline
King Abdullah University of Science and Technology (KAUST)
\newline 
Thuwal 23955-6900,  Saudi Arabia
}
\email{athanasios.tzavaras@kaust.edu.sa}

\begin{abstract}

We consider nonlinear viscoelastic materials of Kelvin-Voigt type with stored energies satisfying an Andrews-Ball condition, allowing 
for non convexity in a compact set.  Existence of weak solutions with deformation gradients in $H^1$ is established for energies of any
superquadratic growth. In two space dimensions, weak solutions notably turn out to be unique in this class. 
Conservation of energy for weak solutions in two and three dimensions, as well as global regularity for 
smooth initial data in two dimensions are established under additional mild restrictions on the growth of the stored energy.
\end{abstract}

\maketitle

\section{Introduction}
We consider the Cauchy problem for viscoelastic materials of the strain-rate type in Lagrangean coordinates 
\begin{equation}\label{eq:wave}
\partial_{tt}y-\dive (S(\nabla y))-\Delta \partial_{t}y=0\, , 
\end{equation}
where $y : (0,T) \times \T \to \R^d$, where $\T$ is the $d$-dimensional torus, $d=2,3$ and $T>0$ is arbitrary but finite, and with  initial data 
\begin{equation*}
\begin{aligned}
&y|_{t=0}=y_0,\quad & \partial_{t}y|_{t=0}=v_0 \, .
\end{aligned}
\end{equation*}

This second-order system describes motions of a viscoelastic material of strain-rate type with the Piola-Kirchhoff stress
tensor 
\begin{equation}
\label{eq:constve}
T_R = S(F) + \partial_{t}F \, , \quad S(F) = \frac{\del W}{\del F}(F) \, .
\end{equation}
It is assumed throughout that the elastic part of the stress is given as the gradient of a strain-energy function, $S= DW$,
while the viscous part of the stress is linear, leading to \eqref{eq:constve}. Such constitutive relations fit under the general framework of viscoelasticity
of strain-rate type, and specifically into the class of Kelvin-Voigt type materials; we refer to  Antman \cite[Ch 10, Secs 10-11]{Antman05}
and Lakes \cite[Ch 2]{Lakes09} for general information on viscoelasticity and the specific terminologies. 

In this work, we focus on the model \eqref{eq:constve} with linear dependence on the strain-rate and study the effect of nonlinear elastic behavior on various aspects of the existence theory. Note that the constitutive relation \eqref{eq:constve} violates frame-indifference and we refer the reader to the end of the Introduction for a discussion.


The system \eqref{eq:wave} is expressed as a hyperbolic-parabolic system,
\begin{equation}\label{eq:main}
\begin{aligned}
\partial_{t}v-\dive(S(F))-\Delta\,v&=0\\
\partial_{t}F-\nabla\,v&=0\\
\curl\,F&=0\, ,
\end{aligned}
\end{equation}	
for the functions  $v:(0,T)\times\T\to\R^d$, $F :(0,T)\times\T\to M^{d\times d}$, where  $v = \partial_{t}y$ stands for the velocity
and  $F =\nabla y$ for the deformation gradient. System \eqref{eq:main} is supplemented with periodic boundary conditions and initial data
 at $\{t=0\} \times \T$ given by
 \begin{equation}\label{eq:id}
\begin{aligned}
&v|_{t=0}=v_0, \\
& F|_{t=0}=F_0  = \nabla y_0 \, .
\end{aligned}
\end{equation}
The constraint $\curl\,F=0$ is an involution of the dynamics and it is propagated from the initial data
 $F_0 = \nabla y_0$ by the kinematic compatibility equation $F_t = \nabla v$.

A number of studies concern the issue of existence for viscoelastic models of strain-rate type both for
linear viscosity \cite{Pego87, AB, DF} as well as nonlinear strain-rate dependence \cite{FN88, Demoulini00, Tvedt08, BMR12};
the interested reader may find a thorough review of previous literature in \cite{DF}. 
Our objective is to study the effect of nonlinear elastic response on the existence, uniqueness, and
regularity theory of \eqref{eq:main}. 
We limit attention to \eqref{eq:constve} in several space dimensions, 
and in that sense the most relevant studies are the existence theory of Friesecke-Dolzmann \cite{DF}, and the works concerning
convex stored energies  of Friedman-Ne\v{c}as \cite[Sec 6]{FN88} and Engler \cite{En1,En2}. 
There are available studies for models with nonlinear strain-rate dependence $T_R = S(F,  F_t)$, 
under monotonicity hypotheses for $S(F, \cdot)$ and globally Lipschitz hypotheses that restrict to linear growth in $F$, see 
\cite{FN88, Demoulini00, Tvedt08}.

 In our analysis, we allow for nonconvex strain energy functions $W(F)$ with growth conditions of polynomial type. 
Instead of convexity  we adopt
 the Andrews-Ball condition  \cite{AB, DF} imposing monotonicity at infinity;
namely,  $W(F)$ satisfies for some $R>0$ 
\begin{equation}\label{eq:ab}\tag{AB}
(S(F_1)-S(F_2), F_1-F_2)\geq 0,  \qquad \forall \;  |F_1|, |F_2| \ge R \, ,
\end{equation}
where $(F,G) = {\rm tr}\,FG^T$ denotes the inner product on $\R^{d\times d}$.
On occasion a strengthened version of \eqref{eq:ab} is employed, requesting that there exist constants $C>0$, $R>0$ such that 
\begin{equation}\label{eq:abprime}\tag{AB$^\prime$}
(S(F_1)-S(F_2) , F_1-F_2)\geq   C(|F_1|^{p-2}+|F_2|^{p-2})  |F_1-F_2|^2, \qquad \forall \;  |F_1|, |F_2| \ge R \, .
\end{equation}
Condition \eqref{eq:ab} amounts to requiring the strain energy to be a semiconvex function while, similarly, 
\eqref{eq:abprime} implies a strenghtened  variant of semiconvexity; see \cite{DF} and Lemma \ref{lem:modAB} below.

A-priori bounds for  the system \eqref{eq:main} are provided by the energy identity
\begin{equation*}
\int \tfrac{1}{2} |v|^2 + W(F) \, dx + \int_0^t \int    |\nabla v|^2 \, dx = \int \tfrac{1}{2} |v_0|^2 + W(F_0) \, dx \, ,
\end{equation*}
but clearly they do not suffice to yield an existence theory.
  Friesecke and Dolzmann\cite{DF} in a penetrating study provide global existence of weak solutions for  \eqref{eq:wave} 
under the Andrews-Ball condition \eqref{eq:ab}. Key to their analysis is the property of propagation of compactness for the deformation gradient, \cite[Prop 3.1]{DF}; the latter is complemented with a variational time-discretization scheme 
to achieve existence of weak solutions in energy space. We provide here an example indicating that the system does not enjoy any compactification mechanism, by showing that oscillations in the initial strain can persist in the dynamics.
The example  concerns one dimensional models for phase transitions, 
\begin{equation}
\label{eq:onedvisco}
\begin{aligned}
u_t &= v_x
\\
v_t &= \sigma(u)_x + v_{xx}\, ,
\end{aligned}
\end{equation}
and combines the universal class of uniform shearing solutions with the observation of Pego \cite{Pego87} and Hoff \cite{Hoff87} that
\eqref{eq:onedvisco} admits solutions with discontinuities in the strain and strain-rate. Combining these ingredients,
exact oscillatiory solutions are constructed for \eqref{eq:onedvisco} with non-monotone stress-strain laws. 
Our example corroborates an example of Friesecke and Dolzmann at the level of  approximating solutions for \eqref{eq:wave}, \cite[Example 2.1]{DF},
and indicates that persistent oscillations in the strain is a feature of the viscoelasticity system. 

The property of propagation of compactness can also be seen as  propagation of regularity, which in \cite{DF} is the propagation of the $L^2$-modulus of continuity. In the current work, a crucial observation is that in fact  $H^{1}$-regularity of the initial deformation gradient is also propagated.
Precisely, the following energy bound holds
\begin{equation}
\label{eq:unicab}
\frac{d}{dt} \int   |  v - \tfrac{1}{2} \dive F|^2 + \tfrac{1}{2} |\nabla F |^2+  2 W(F) dx 
+  \int  \Big (  D^2 \tilde W : (\nabla F, \nabla F) +  |\nabla v|^2 \Big ) dx \le K \int |\nabla F|^2 dx\, ,
\end{equation}
for $\tilde W(F):=W(F)+K\frac{|F|^2}{2}$, which is convex because of the condition \eqref{eq:ab}. Then, for data $(v_0,F_0)\in L^{2}(\T)\times H^{1}(\T)$,  we establish
an existence theory for weak solutions $(v,F)$ of \eqref{eq:main}-\eqref{eq:id} such that 
\begin{equation}
\label{classR}
\begin{aligned}
&v\in L^{\infty}(0,T;L^{2}(\T))\cap L^{2}(0,T;H^{1}(\T))\, ,\\
&F\in L^{\infty}(0,T;H^{1}(\T)\cap L^{p}(\T))\, ;
\end{aligned}
\end{equation}
see Theorem \ref{teo:ex}. \par
The propagation of $H^1$-regularity of the deformation gradient has several important implications. First, we prove that solutions satisfying \eqref{classR} conserve the energy in $d=2$ and, if the non-linearity $S$ does not grow too fast, in  $d=3$ as well, see Theorem \ref{teo:energy}. Second, a main consequence of the $H^1$-estimation are the uniqueness and regularity properties for solutions in two dimensions.
Uniqueness results for \eqref{eq:main} or related systems are provided in  \cite{DF, Demoulini00, Tvedt08}, always based on global Lipschitz assumptions for the stress function.  Instead, thanks to the propagation property of $H^1$-regularity, we prove uniqueness of weak solutions $(v,F)$ of class \eqref{classR} for dimension $d=2$ assuming some mild restrictions on the growth of $D^2W$, and thus substantially improving 
upon the global Lipschitz assumptions on $S(F)$.
Indeed, uniqueness is established in Theorem \ref{teo:uni} under hypothesis \eqref{eq:ab} with the growth of $W$ restricted to  $2\leq p < 4$, or 
in Theorem \ref{teo:unip} under the strengthened condition \eqref{eq:abprime} for any growth $p\geq 2$.
Moreover, assuming that $W(F) \sim |F|^p$ for $|F|$ large and using classical energy estimates, based on the Gagliardo-Nirenberg and Brezis-Gallouet inequalities, we establish smoothness of solutions under conditions \eqref{eq:ab} and $2\le p \le 5$ or \eqref{eq:abprime} and $2\le p \le 6$; see Theorem \ref{teo:reg}.

These results highlight a striking analogy with the $2D$ incompressible Euler equations, at least when the non-linearity $S(F)$ does not grow too fast. Indeed, for $2D$ Euler, existence holds in the class of solutions with vorticity in $L^{q}$ with $q\leq\infty$, in analogy to Theorem \ref{teo:ex}.  Moreover, existence and uniqueness holds for solutions with bounded vorticity, in analogy to Theorems \ref{teo:uni} and \ref{teo:unip}, and finally for both systems there is global regularity given smooth initial data, in analogy to Theorem \ref{teo:reg}. We also remark that besides the conceptual similarities just mentioned, the uniqueness and global regularity proofs exploit critical (logarithmic) estimates similar to Yudovich \cite{Y} and Beale, Kato and Majda \cite{BKM}.

Some related material is listed in two appendices.  The key estimate \eqref{eq:unicab}, used in the existence proof, 
can be understood by considering the problem of transfer of dissipation.  
It was noted by DiPerna \cite{DiPerna83} that the estimate on velocity gradients  obtained by energy dissipation in viscoelasticity models
can be transferred to strain gradients; a similar observation holds for relaxation systems \cite{Tzavaras99}. 
Such estimates are here established for several space dimensions, see the modulated energy estimate \eqref{unic}, and is what lies behind 
the existence Theorem \ref{teo:ex}.
Another class of approximations used for elasticity is the so-called diffusion-dispersion approximations that occur when viscosity is
intermixed with higher-order gradient theories. An observation of Slemrod \cite{Slemrod89} indicates that  diffusion-dispersion approximations of one-dimensional elasticity can be transformed to parabolic approximations.
We show that this is also the case in several space dimensions.

\subsection*{Frame Indifference}\label{subsec:1}
 We conclude this introduction with a short discussion on frame-indifference. The constitutive theory of (isothermal) 
viscoelasticity of strain-rate type asserts that the Piola-Kirchhoff stress $T_R = T_R (F, \dot F)$ depends on the deformation gradient $F$ and the strain rate $\dot F$.  The stress is nominally decomposed as 
\begin{equation}
\label{eq:ctv}
T_R = T_{el} (F) + T_v (F, \dot F)\, , 
\end{equation}
 where $\dot\,$ denotes time derivative, $T_{el}$ is the elastic part of the stress, $T_v$ the viscous part, and $T_v (F, 0) = 0$. 
Compatibility with the Clausius-Duhem inequality dictates,
\begin{equation}
\label{eq:cd}
T_{el} (F) = \frac{\del W}{\del F} (F) \, , \quad T_v (F, \dot F) : \dot F \ge 0 \, ,
\end{equation}
that the elastic part is induced by a strain energy (or stored energy) function $W(F)$ while the viscous part is dissipative.
The terminology Kelvin-Voigt model originates from the interpretation of viscoelastic behaviour through systems of spring and dashpot mechanisms,
Lakes \cite[Ch 2]{Lakes09}, 
and Kelvin-Voigt specifically refers to additive decomposition of the elastic and viscous stresses.

In the continuum mechanics literature the principle of material frame indifference is imposed on constitutive theories,
which posits that two observers moving with respect to each other with a Euclidean transformation,
\[
x^* = Q(t) x + d(t), \, \quad t^* = t + a\, , 
\]
with $Q(t)$ an arbitrary proper orthogonal tensor, $Q^T Q= Q Q^T = I$ and $\det Q = 1$, should observe the same constitutive relations in their respective frames of reference \cite{TN}.
When applied to the constitutive theory of viscoelasticity of strain-rate type,  material frame indifference
implies that the Piola-Kirchhoff stress tensor $T_R$ must satisfy 
$T_R(F,\dot F) = F G(C,\dot C)$  for some symmetric tensor-valued function $G$ of $C=F^TF$, 
{\it e.g.} Antman \cite[Ch 10, Secs 10-11]{Antman05}, \cite{Sengul10}.  For  \eqref{eq:ctv}-\eqref{eq:cd} this suggests
\begin{equation*}
T_{el} (F) = 2 F \, \frac{\del \varphi}{\del C} (C) \, , \quad T_v (F, \dot F) = F G_v (C, \dot C)\, , 
\end{equation*}
where $\varphi(C) = W(\sqrt{C})$ and $G_v^T (C, \dot C) = G_v (C, \dot C)$. 

For reasons related to analysis it is beneficial to impose a strengthened version of \eqref{eq:cd} on the
energy dissipation,
assuming that for some $\gamma > 0$
\begin{equation}
\label{eq:sdiss}
T_v (F, \dot F) : \dot F  \ge \gamma |\dot F|^2 \, .
\end{equation}
We outline an argument of {\c S}eng\"ul \cite{Sengul10} showing that \eqref{eq:sdiss} is incompatible with frame indifference.
Indeed, using the symmetry of $G_v$,
\[
T_v (F, \dot F) : \dot F = F G_v (C, \dot C)  : \dot F = G_v (C, \dot C) : F^T \dot F = \frac{1}{2}  G_v (C, \dot C) : \dot C
\]
and \eqref{eq:sdiss} becomes
\[
\frac{1}{2}  G_v (C, \dot C) : \dot C \ge \gamma |\dot F|^2\, .
\]
The latter is violated by the example $F(t) = e^{t \Omega}$ with $\Omega^T = - \Omega$ a constant skew-symmetric matrix,
for which $C(t) = I$, $\dot F = \Omega e^{t \Omega}$ and $|\dot F (t) | = |\Omega| \ne 0$.
This argument indicates that  the  Kelvin-Voigt model in Lagrangian coordinates $T_R(F,\dot F) = S(F) + \dot F$ violates frame-indifference, as also do any reasonable linear viscoelastic models as first noted by Antman \cite{antman}.

The principle of material frame indifference is strictly speaking a hypothesis  imposed on the form of constitutive relations of 
continuum physics. It reflects the intuition that stress results from deformations originating from an unstrained state and it should not be affected
by the superposition of an arbitrary rigid body motion. 
Detractors argue that constitutive relations reflect microscopic dynamics determined via Newton's laws which are only
Galilean invariant. According to this viewpoint, frame indifference might be too restrictive 
and should be replaced by invariance under the Galilean group or the extended Galilean group. 
(The latter only requires that rotations $Q$ are time independent, and in particular this invariance admits the model \eqref{eq:constve}.)
The reader is referred to the very interesting review by Speziale \cite{Speziale98} (and references therein) 
which tests the validity of frame indifference of constitutive relations, 
in a context where fluctuations result from kinetic modeling or from models for turbulence.

%

\subsection*{Organization of the paper} The paper is organized as follows. Section \ref{sec2} lists the hypotheses on the stored energy, discusses their interrelations, and contains the statements of the main results.  Section \ref{sec:tech} contains auxiliary results that are needed later in the proofs. Section \ref{sec4} contains the
proof of the existence of weak solutions, which is based on a  Galerkin approximation and
a compactness argument for the constructed Galerkin iterates using the Aubin-Lions-Simon Lemma.  Section \ref{sec5} contains the uniqueness proof in two space dimensions, while Section \ref{sec6}
the proof of global regularity. Then, in Section \ref{sec:sust} we provide the construction of examples
of one-dimensional oscillating solutions in the nonlinear case with phase transitions, and then in the linear case. Appendix \ref{AppA} contains the
energy estimate indicating the transfer of dissipation, while Appendix \ref{AppB} shows how the multi-dimensional diffusion-dispersion
approximation of the elasticity system can be transformed to a parabolic approximation of conservation laws.  Finally, in Appendix \ref{AppC} we discuss the assumptions on the stored energy, in particular the ones regarding its growth at infinity.

\section{Assumptions on the stored energy and main results}\label{sec2}
In this section we fix the the assumptions on the stored energy $W$ and we present all our main results. 
\subsection{Hypotheses on the stored energy}
We always assume throughout the paper that the stored energy $W$ satisfies for some $p\geq 2$ the following hypotheses:
\begin{itemize}
\item[(A1)] $W\in C^{2}(\R^{d\times d};\R)$.
\item[(A2)] There exists $c>0$ such that 
\begin{equation*}
c(|F|^p-1)\leq W(F).
\end{equation*}
\item[(A3)] There exists $C>0$ such that 
\begin{equation*}
\begin{aligned}
&|W(F)|\leq C(1+|F|^p);\,&|DW(F)|\leq C(1+|F|^{p-1}).\\
\end{aligned}	
\end{equation*}
\end{itemize} 
The following assumption characterizes the dissipative nature of the stored energy:
\begin{itemize}
\item[(A4)] There exists constant $K>0$ such that 
\begin{equation*}
(S(F_1)-S(F_2), F_1-F_2)\geq -K|F_1-F_2|^2.
\end{equation*}
\end{itemize}
In some cases, as mentioned in the Introduction, we consider the following more restrictive assumption:
\begin{itemize}
\item[(A4$^\prime$)]	There exists $C> 0$ and $K>0$ such that for any $F_1$ and $F_2$ we have that 
\begin{equation}\label{eq:A4}
(S(F_1)-S(F_2) , F_1-F_2)\geq   \big ( C (|F_1|^{p-2}+|F_2|^{p-2})-K \big ) |F_1-F_2|^2.
\end{equation}
\end{itemize}
In the next lemma we prove that (A4) and (A4$^\prime$) are direct consequences of \eqref{eq:ab} and \eqref{eq:abprime}, respectively. 
They imply that upon adding a quadratic function of $F$ the stored energy becomes convex.
The latter condition is called semiconvexity, it has geometric implications on the graph of $W$, and
is used extensively in the theory of Hamilton-Jacobi equations.
\begin{lemma}\label{lem:modAB}
Assume that $W$ satisfies (A1)-(A3). Then:
\begin{enumerate}
\item If $W$ satisfies \eqref{eq:ab}, then $W$ satisfies (A4) and 
\begin{equation*}
\tilde{W}(F)=W(F)+\frac{K}{2}|F|^2
\end{equation*}
is convex. In particular, $D^2\tilde{W}\geq 0$.
\item If $W$ satisfies \eqref{eq:abprime}, then $W$ satisfies (A4$^{\prime}$) and
\[
\tilde{W}(F)=:W(F)+\frac{K}2|F|^2
\]
satisfies 
\begin{equation*}
D^2\tilde{W}(F)\geq  c |F|^{p-2}\mathbb{I}.
\end{equation*}
\end{enumerate}
\end{lemma}
\begin{proof}
Note that {\em (1)} is proved in \cite[Lemma 1.1]{DF}. Regarding {\em (2)}, and following the ideas in the proof of \cite[Lemma 1.1]{DF}, let $R>0$ be the radius appearing in \eqref{eq:abprime} and consider the balls $B_R(0)$ and $B_{2R}(0)$. If $F_1$, $F_2\notin B_{R}(0)$, then the stronger assertion \eqref{eq:abprime} holds, whereas if $F_1$, $F_2\in B_{2R}(0)$ we find that
\begin{align*}
\left|(S(F_1)-S(F_2) , F_1-F_2) -   C (|F_1|^{p-2}+|F_2|^{p-2})  |F_1-F_2|^2 \right| & \leq K_1 |F_1 - F_2|^2 + 2C(2R)^{p-2}|F_1 - F_2|^2\\
& \leq K |F_1 - F_2|^2,
\end{align*}
where $K_1>0$ is the Lipschitz constant of $S$ on the ball $B_{2R}(0)$. Hence \eqref{eq:A4} holds and we are left to consider the case $F_1\notin B_{2R}(0)$, $F_2\in B_R(0)$, since the remaining case amounts to switching the roles of $F_1$ and $F_2$. Define 
\[
F(\lambda) = F_1 - \lambda \frac{F_1 - F_2}{|F_1 - F_2|}.
\]
Since $F(0) = |F_1| > 2R$ and $|F(|F_1-F_2|)| = |F_2| < R$, there exists $\lambda_0\in (0, |F_1 - F_2|)$ such that
\[
|F(\lambda_0)| = \frac{3R}{2}.
\]
Write $F_0 := F(\lambda_0)$ and note that
\begin{equation}\label{eq:lambda0}
\frac{F_1 - F_0}{\lambda_0} = \frac{F_1 - F_2}{|F_1 - F_2|}.
\end{equation}
We may thus compute that
\begin{align}
(S(F_1)-S(F_2) , F_1-F_2) & = (S(F_1)-S(F_0) , F_1-F_2) + (S(F_0)-S(F_2) , F_1-F_2) \nonumber\\
& = (S(F_1)-S(F_0) , F_1-F_0) \frac{|F_1 - F_2|}{\lambda_0} + (S(F_0)-S(F_2) , F_1-F_2) \nonumber \\
&\geq C (|F_1|^{p-2}+|F_0|^{p-2})  |F_1-F_0|^2  \frac{|F_1 - F_2|}{\lambda_0} \nonumber\\
&\qquad - K_1|F_0 - F_2||F_1 - F_2|,\label{eq:lemmaAB0}
\end{align}
where the last inequality follows from \eqref{eq:abprime} and the fact that $F_1$, $F_0\notin B_R(0)$, whereas $F_0$, $F_2\in B_{2R}(0)$ and $K_1$ denotes again the Lipschitz constant of $S$ on $B_{2R}(0)$. Next, note that
\begin{align}\label{eq:lemmaAB1}
|F_0 - F_2| & \leq \frac{5R}{2} = \frac{5R}{2} \frac{1}{|F_1 - F_2|}|F_1 - F_2| \leq \frac{5}{2} |F_1 - F_2|,
\end{align}
since $|F_1 - F_2|>R$. Moreover, using \eqref{eq:lambda0}, we find that
\begin{align}
\frac{|F_1 - F_0|^2}{\lambda_0} |F_1 - F_2| & = |F_1 - F_0||F_1 - F_2|  = |F_1 - F_2|^2 \frac{|F_1 - F_0|}{|F_1 - F_2|} \label{eq:lemmaAB2}
\end{align}
and we aim to prove that
\begin{equation}\label{eq:lemmaAB3}
\frac{|F_1 - F_0|}{|F_1 - F_2|} \geq L > 0.
\end{equation}
Indeed, as $F_0$, $F_2\in B_{2R}$,
\[
\lim_{|F_1|\to\infty} \frac{|F_1 - F_0|}{|F_1 - F_2|} = 1
\]
and thus there exists $M>0$ such that whenever $|F_1|>M$ it holds that
\[
\frac{|F_1 - F_0|}{|F_1 - F_2|} > 1 -\varepsilon.
\]
On the other hand, when $|F_1|<M$, we find that
\[
\frac{|F_1 - F_0|}{|F_1 - F_2|} \geq \frac{R/2}{M+R}
\]
proving \eqref{eq:lemmaAB3}. Using \eqref{eq:lemmaAB1}-\eqref{eq:lemmaAB3}, and noting that $|F_0|> |F_2|$ we find that \eqref{eq:lemmaAB0} implies 
\[
(S(F_1)-S(F_2) , F_1-F_2) \geq L C (|F_1|^{p-2}+|F_2|^{p-2}) |F_1 - F_2|^2 - \frac{5}{2}K_1|F_1 - F_2|^2
\]
which is \eqref{eq:A4}. To conclude the proof of Lemma \ref{lem:modAB}, we are left to show that \eqref{eq:A4} implies $D^2\tilde{W}(F)\geq  c |F|^{p-2}\mathbb{I}$. Indeed, apply \eqref{eq:A4} to $F_1 = F$, $F_2 = F + t H$ where $t > 0$ and $H$ a matrix to obtain
\[
\Big ( \frac{S( F + t H)-S(F)}{t}  , H  \Big ) \geq   \big [ c (|F + t H |^{p-2}+|F|^{p-2})-K \big ]|H|^2.
\]
Taking $t \to 0$ gives
\[
\Big ( D^2 W(F) H  , H  \Big ) \geq   \big ( 2c |F|^{p-2} - K \big )|H|^2
\]
and implies the result.
\end{proof}

\subsection{Main results}



We start by providing the definition of a weak solution to the Cauchy problem \eqref{eq:main}-\eqref{eq:id}.

\begin{definition}\label{def:ws}
Let $v_0\in L^{2}(\T)$ and $F_0\in L^{p}(\T)$. The pair 
\begin{equation}\label{eq:ws1}
(v,F)\in L^{\infty}(0,T;L^2(\T))\cap L^{2}(0,T;H^1(\T))\times L^{\infty}(0,T;L^{p}(\T))
\end{equation}
is a weak solution to the initial value problem \eqref{eq:main}
if the following are satisfied:
\begin{itemize}
\item For any $\phi, \psi\in C^{\infty}_c([0,T);C^{\infty}(\T))$ 
\begin{equation}\label{eq:ws2}
\begin{aligned}
&\iint\,v\psi_t+S(F)\nabla\psi+\nabla\,v\nabla\psi\,dxdt-\int\,v_0(x)\psi(0,x)\,dx=0\\
&	\iint\,F\phi_t+\nabla\,v\phi\,dxdt-\int\,F_0(x)\phi(0,x)\,dx=0\\
& F(t,x)=\nabla\,y(t,x)\mbox{ for a.e. }(t,x)\in (0,T)\times\T
\end{aligned}	
\end{equation} 
\item The energy inequality holds for a.e. $t\in(0,T)$
\begin{equation}\label{eq:ws3}
\int\,\frac{|v|^2}{2}+W(F)\,dx+\int_0^t\int\,|\nabla\,v|^2\,dxds\leq \int\,\frac{|v_0|^2}{2}+W(F_0)\,dx.
\end{equation}	
\end{itemize}
\end{definition}
The first main result of the paper is the following: 
\begin{theorem}\label{teo:ex}
Assume that $W$ satisfies (A1)-(A4) for some $p\geq 2$. Then, for any initial data $(v_0,F_0)$ with $F_0=\nabla\,y_0$ a.e. in $\T$ and 
\begin{equation}\label{eq:hypid}
(v_0, F_0)\in L^{2}(\T) \times L^{p}(\T),
\end{equation}
there exists at least one weak solution in the sense of Definition \ref{def:ws}. Moreover, 
\begin{enumerate}
\item if $F_0\in H^{1}(\T)$ then $F$ satisfies
\begin{equation*}
F\in L^{\infty}(0,T;H^{1}(\T)).
\end{equation*}
\item if $F_0\in H^{1}(\T)$ and (A4$^\prime$) holds, then $F$ satisfies
\begin{equation}\label{eq:h1estp}
\begin{aligned}
&F\in L^{\infty}(0,T;H^{1}(\T))\\
&\nabla|F|^{\frac{p}{2}}\in L^{2}(0,T;L^2(\T)).
\end{aligned}
\end{equation}
\end{enumerate}
\end{theorem}

\begin{remark}
We note that the global existence for initial data $(v_0, F_0)\in L^{2}(\T) \times L^{p}(\T)$ was already proved in \cite{DF}. Instead and to the best of our knowledge, $\it (1)$ and $\it (2)$ are new contributions.
\end{remark}

The next result concerns the conservation of energy. 
\begin{theorem}\label{teo:energy}
Assume that $W$ satisfies (A1)-(A3). For $v_0\in L^{2}(\T)$ and $F_0\in (H^{1}\cap L^{p})(\T)$ let $(v,F)$ be a weak solution in the sense of Definition \ref{def:ws}. Then, if $F\in L^{\infty}(0,T;H^{1}(\T))$,  with $d=2$ and $p\geq 2$ or with $d=3$ and $2\leq p\leq 4$, the weak solution $(v,F)$ verifies for any $t\in[0,T]$
\begin{equation*}
\int\,\frac{|v|^2}{2}+W(F)\,dx+\int_0^t\int\,|\nabla\,v|^2\,dxds=\int\,\frac{|v_0|^2}{2}+W(F_0)\,dx.
\end{equation*}	
\end{theorem}

\begin{remark}\label{rem:energy}
The regularity hypothesis  $F\in L^{\infty}(0,T;H^{1}(\T))$ suffices to guarantee  conservation of energy. On the other hand, Hypothesis (A4) is needed in Theorem \ref{teo:ex} for the existence of weak solutions in that class.
It is interesting to note that assuming (A4$^{\prime}$) instead of (A4) does not seem to improve the range of $p$ in the three-dimensional case.
\end{remark}

The next results concern the uniqueness of weak solutions in the two-dimensional case. We additionally assume that the second derivatives of $W$ have polynomial growth. Precisely, we assume that there exists $C>0$ such that 
\begin{equation}\label{eq:growthsecond}
\begin{aligned}
&|D^2W(F)|\leq C(1+|F|^{s}),\mbox{ with }s\geq p-2.
\end{aligned}
\end{equation}
In particular, the lower bound on $s$ in \eqref{eq:growthsecond} is needed for consistency with the coercivity assumption (A2); see Lemma \ref{lem:app1}. The first result concerning the uniqueness deals with assumption $(A4)$.
\begin{theorem}\label{teo:uni}
Assume that $d=2$ and let $\gamma\in[0,2]$. Assume that  $2\leq p<2+\gamma$ and $W$ satisfies (A1)-(A4) and \eqref{eq:growthsecond} with $p-2\leq s\leq p-\gamma$. Then, for any initial data $(v_0,F_0)$ with $F_0=\nabla\,y_0$ for a.e. in $\Tt$ and
\begin{equation*}
(v_0, F_0)\in L^{2}(\Tt) \times H^{1}(\Tt),
\end{equation*}
there exists a unique weak solution such that 
\begin{equation}\label{eq:h1est1}
\begin{aligned}
&F\in L^{\infty}(0,T;H^{1}(\Tt)).
\end{aligned}
\end{equation}
\end{theorem}
A more satisfactory result can be obtained by invoking assumption (A4$^{\prime}$).  The statement of uniqueness in this case follows in Theorem \ref{teo:unip}.

\begin{theorem}\label{teo:unip}
Assume that $p\geq 2$ and $d=2$ and $W$ satisfies (A1)-(A4$^{\prime}$) and \eqref{eq:growthsecond} for some $p-2\leq s<p$. Then, for any initial data $(v_0,F_0)$ with $F_0=\nabla\,y_0$ a.e. in $\Tt$ and 
\begin{equation*}
(v_0, F_0)\in L^{2}(\Tt) \times H^{1}(\Tt),
\end{equation*}
there exists a unique weak solution such that 
\begin{equation}\label{eq:h1est1p}
\begin{aligned}
&F\in L^{\infty}(0,T;H^{1}(\Tt))\\
& \nabla|F|^{\frac{p}{2}}\in 
L^{2}(0,T;L^{2}(\Tt)).
\end{aligned}
\end{equation}
\end{theorem}
\begin{remark}
As in Remark \ref{rem:energy}, the proofs of  uniqueness in Theorems \ref{teo:uni} and \ref{teo:unip} use only the regularity in \eqref{eq:h1est1} and \eqref{eq:h1est1p}. On the other hand, in view of Theorem \ref{teo:ex}, (A4) (resp.\ (A4$^{\prime}$)) guarantee the existence of weak solutions of
regularity class \eqref{eq:h1est1} (resp.\ \eqref{eq:h1est1p}).  
\end{remark}
The last main result in the present article concerns global regularity. We assume that $W(F) $ behaves  like $|F|^p$ for $|F|$ large. More precisely, we assume that $W\in C^{4}(\R^{2\times2};\R)$ and there exists $C>0$ such that  
\begin{equation}\label{eq:addhyp}
\begin{aligned}
&|D^2W(F)|\leq C(1+|F|^{p-2}),\,\,|D^3W(F)|\leq C(1+|F|^{p-3}),\,\,\,3< p,\\
&|D^2W(F)|\leq C(1+|F|^{p-2}),\,\,|D^3W(F)|\leq C,\,\,\,2< p\leq 3,\\
&|D^2W(F)|\leq C,\,\,|D^3W(F)|\leq C,\,\,\,p=2,
\end{aligned}
\end{equation}
and that also the fourth derivatives have a polynomial growth without a precise order. 
\begin{theorem}\label{teo:reg}
Assume $d=2$ that $W\in C^{4}(\R^{2\times2};\R)$ satisfies (A2), (A3), and \eqref{eq:addhyp}. Then, for any initial data $(v_0,F_0)$ with $F_0=\nabla\,y_0$ a.e. in $\Tt$ and
\begin{equation*}
(v_0,F_0)\in H^{3}(\Tt)\times H^{3}(\Tt),
\end{equation*}
the unique weak solution of \eqref{eq:main}-\eqref{eq:id} constructed in Theorem \ref{teo:uni} satisfies
\begin{equation*}
(v,F)\in L^{\infty}(0,T;H^{3}(\Tt))\times L^{\infty}(0,T;H^{3}(\Tt)),
\end{equation*}
provided $W$ satisfies (A4) and $2\leq p\leq 5$ or $W$ satisfies (A4$^\prime$) and $2\leq p\leq 6$.
\end{theorem}
%
%
%

\section{Some technical lemmas}\label{sec:tech}

In this section we recall some classical technical lemmas which play a crucial role in the proofs of our main results.
The first lemma contains some classical interpolation inequality. First, we recall the Gagliardo-Nirenberg-Sobolev interpolation inequality, the critical Sobolev embedding inequality, and we note that the precise constants in {\it (1)} and {\it (2)} are well-known and can be deduced for example from the result in \cite{KW}, see also \cite[Theorem 8.5]{LL} for{\it (2)}. Moreover, in the lemma below we also include a version of the Brezis-Gallouet inequality which can easily be deduced from the original statement in \cite{BG}. The notation $\|\cdot\|_p$ , $p\in[1,\infty]$, denotes the norm of the classical Lebesgue spaces $L^{p}(\T)$. 

\begin{lemma}\label{lem:gani}
Let $f\in H^{1}(\Tt)$, then the following inequalities hold:
\begin{enumerate}
\item For any $r>1$ and $q>1$, there exists a constant $C_r>0$ such that 
\begin{equation*}
\|f\|_{\frac{rq}{q-1}}\leq C_r\|f\|_r^{1-\frac{1}{q}}\|f\|_{H^1}^{\frac{1}{q}}.
\end{equation*}
In particular, for $r=2$ and $q=2$ it holds that
\begin{equation*}
\|f\|_{4}	\leq C\|f\|_2^{\frac{1}{2}}\|f\|_{H^1}^{\frac{1}{2}}.
\end{equation*}
\item There exists a constant $C>0$ such that for any $q>2$
\begin{equation*}
\|f\|_q\leq C\sqrt{q}\|f\|_{H^1}.
\end{equation*}
\item If, in addition, $f\in H^2(\Tt)$, then there exists a constant $C>0$ such that 
\begin{equation*}
\|f\|_{\infty}\leq C(1+\|f\|_{H^1})\left(\log\left(e+\|f\|_{H^2}\right)\right)^{\frac{1}{2}}.
\end{equation*}
\end{enumerate}
\end{lemma}
Next, we recall the classical result concerning the maximal $L^p$ regularity for the heat equation on the torus. The result on the entire space can be found in \cite[Chapter IV, Section 3]{LSU}. Here and based on this result, we provide a short proof for the torus, which is also classical but difficult to find in the literature. 
\begin{lemma}\label{lem:maxlp}\mbox{}
For a smooth function $G$, let $u$ be a smooth solution of the following initial value problem:
\begin{equation}\label{eq:heattorus}
\begin{aligned}	
\partial_t u-\Delta u&=\dive G\mbox{ in }(0,T)\times\T\\
u|_{t=0}&=0\mbox{ on }\{t=0\}\times\T.
\end{aligned}
\end{equation}
Then, for any $t\in(0,T)$ and for any $1<r<\infty$
\begin{equation}\label{eq:partorus}
\int_0^t\|\nabla u\|_r^r\,ds\leq C\int_0^t\|G\|_r^r\,ds,
\end{equation}
where $C=C(r,d)$.
\end{lemma}
\begin{proof}
We first recall that if $\tilde{G}$ and $\tilde{f}$ are in $C^{\infty}_c((0,T)\times\R^d)$ and $\tilde{u}\in C^{\infty}((0,T)\times\R^d))$ satisfy
\begin{equation*}
\begin{aligned}
\partial_t\tilde u-\Delta\tilde{u}&=\dive \tilde{G}+\tilde f\\
\tilde{u}|_{t=0}&=0
\end{aligned}
\end{equation*}
then it follows that 
\begin{equation}\label{eq:parwholespace}
\int_0^T\|\nabla \tilde u\|_{L^r(\R^d)}^{r}\,dt\leq C(r,d)\left(\int_0^T\|\tilde G\|_{L^r(\R^d)}^r\,dt+\int_0^T\|\tilde f\|_{L^r(\R^d)}^{r}\,dt\right), 
\end{equation}
see \cite[Chapter IV, Section 3]{LSU}. Next, given $u$, $G$ and $f$ as in \eqref{eq:heattorus}, we extend the functions periodically on the whole $\R^d$ and we denote by $Q^{N}$ the cube parallel to the axes, centered at zero and side-length $2N$. Let 
$\phi_N\in C^{\infty}_{c}(\R^d)$ be a cut-off function such that $\phi_N=1$ on $Q^{N}$ and $\mbox{supp}\phi_N\subset Q^{N+1}$. Note that we can also assume $|\nabla\phi_N|\leq C$ and $|\nabla^2\phi_N|\leq C$. 
Let $u_N:=u\phi_N\in C^{\infty}_c([0,T]\times\R^d)$ and note that $u_N$ solves 
\begin{equation*}
\begin{aligned}
\partial_t u_N-\Delta u_N&=\dive G_N+f_N\\
u_N|_{t=0}&=0
\end{aligned}
\end{equation*}
where $G_N:=G\phi_N$ and $f_N:=-2\nabla u\cdot\nabla\phi_N-u\Delta\phi_N-\nabla\phi_N G$. Since 
$\mbox{supp}\nabla\phi_N$ and $\mbox{supp}\nabla^2\phi_N$ are contained in $Q^{N+1}\setminus Q^{N}$, we infer that 
\begin{equation*}
\begin{aligned}
\int_{0}^T\|f_N\|^r_{L^{r}(\R^d)}\,dt&=\int_{0}^T\|f_N\|^r_{L^{r}(Q^{N+1})}\,dt\\
&\leq C(r)((2N+1)^d-(2N)^d)\left(\int_0^T\|G\|_{r}^r+\|\nabla u\|_r^r+\|u\|_r^r\,dt\right)\\
&:=C(r,d,u)( (2N+1)^d-(2N)^d)
\end{aligned}
\end{equation*}
since there are $( (2N+1)^d-(2N)^d)$ cubes with side-length one in $Q^{N+1}\setminus Q^{N}$. Therefore, by using \eqref{eq:parwholespace}, the definition of $\phi_N$ and the periodicity of $u$ we have 
\begin{equation*}
\begin{aligned}
(2N)^d\int_0^T\|\nabla u\|_r^r\,dt&= \int_0^T \|\nabla u\|_{L^{r}(Q^N)}^r\,dt\leq \int_0^T \|\nabla u_{N}\|_{L^{r}(\R^d)}^r\,dt\\
&\leq C(r,d)\int_0^T\|G_N\|_{L^{r}(\R^d)}^r\,dt+C(r,d)\int_0^T\|f_N\|_{L^{r}(\R^d)}^r\,dt\\
&\leq C(r,d)(2N+2)^d\int_0^T\|G\|_{r}^r\,dt+ C(r,d,u)( (2N+1)^d-(2N)^d).
\end{aligned}
\end{equation*}
Now \eqref{eq:partorus} follows by sending $N\to\infty$. 
\end{proof}

We conclude this section by recalling the following classical generalization of the Gr\"onwall lemma. 
\begin{lemma}\label{lem:bihari}
Let $q>1$ and $C>0$. Suppose $F \in L^1(0,T)$ and that $0\leq y\in C([0,T])$ satisfies the inequality
\begin{equation}\label{eq:bihari}
 y(t) \leq Cq \int_{0}^t F(s) y(s)^{1-\frac{1}{q}}\,ds
\end{equation}
with $y(0)=0$. Then, 
\begin{equation}\label{eq:bihari2}
y(t)\leq\left(C\int_0^t\,F(s)\,ds\right)^{q}.
\end{equation}
\end{lemma}
\begin{proof}
Define $$R(t):= \int_{0}^t F(s) y(s)^{1-\frac{1}{q}}\,ds.$$ By using \eqref{eq:bihari} and the fact that $R$ is absolutely continuous it follows that 
\begin{equation*}
\frac{d}{dt}R(t)\leq F(t)(CqR(t))^{1-\frac{1}{q}}
\end{equation*}
and then 
\begin{equation*}
\frac{d}{dt}R(t)^{\frac{1}{q}}\leq \frac{1}{(Cq)^{\frac{1}{q}}}CF(t).
\end{equation*}
Integrating in time and using the definition of $R$ and \eqref{eq:bihari}, the inequality \eqref{eq:bihari2} follows. 
\end{proof} 

\section{Global existence of weak solutions and conservation of energy}\label{sec4}

In this section we give the proof of Theorem \ref{teo:ex}. We start by defining the approximation system for \eqref{eq:main}-\eqref{eq:id} which is given by a simple Fourier based Galerkin scheme, we refer to \cite{H} for the Galerkin scheme in a general bounded domain $\Omega\subset\R^d$ for the system  \eqref{eq:wave}.

\subsection{The Galerkin Approximation}

Let $N\in\N$ and consider the following initial value problem 
\begin{equation}\label{eq:app}
\begin{aligned}
\partial_t\,v^N-\dive\,P^N(S(F^N))-\Delta\,v^N&=0\\
\partial_t\,F^N-\nabla\,v^N&=0\\
\curl\,F^N &=0,
\end{aligned}
\end{equation}
with initial data
\begin{equation}\label{eq:id2}
\begin{aligned}
&v^N|_{t=0}=v^N_0\\
&F^N|_{t=0}=F^N_0=\nabla\,y^N_0
\end{aligned}
\end{equation}
where the unknowns $(v^N,F^N)$ are defined for $k\in\Z^d$ as
\begin{equation*}
\begin{aligned}
&v^N=\sum_{|k|\leq N}v^N_k(t)e^{ix\cdot k}\\
&F^N=\sum_{|k|\leq N}F^N_k(t)e^{ix\cdot k}= \sum_{|k|\leq N}i\,y^N_k(t)\otimes\,ke^{ix\cdot k}
\end{aligned}
\end{equation*}
with $v^N_k= \bar{v}^N_{-k}$, $F^N_k= \bar{F}^N_{-k}$ and $P^{N}:L^{2}(\T)\to\, P^N(L^2(\T)) $ is the projection operator from $L^{2}(\T)$ to the finite dimensional subspace $P^N(L^2(\T))$ given by the formula
\begin{equation*}
\begin{aligned}
P^N:L^2(\T)&\rightarrow & P^N(L^2(\T))\\
v&\rightarrow & \sum_{|k|\leq N}v_k(t)e^{ix\cdot k}
\end{aligned}
\end{equation*}
where $v_k(t) = \frac{1}{|\T|}\int v(t,x) e^{-ix\cdot k}\,dx$ are the Fourier coefficients of $v$ satisfying $v_k= \bar{v}_{-k}$.

We have the following Proposition. 
\begin{proposition}\label{prop:exgal}
Let $(v_0, F_0)\in L^{2}(\T)\times L^{p}(\T)$ and define $(v_0^N,F_0^N)=(P^Nv_0,P^NF_0)$. Then, for any $N\in\N$ the system  \eqref{eq:app}-\eqref{eq:id2} admits a unique smooth solution.
\end{proposition}
\begin{proof}
We first note that \eqref{eq:app}-\eqref{eq:id2} is equivalent to the following system of ordinary differential equations:
\begin{equation*}
\begin{aligned}
&\frac{d}{dt}v^N_k(t)-i S_k^N(t)k + |k|^2v^N_k(t)=0\\
&\frac{d}{dt}F^{N}_{k}(t)-i v^N_k(t)\otimes k = 0
\end{aligned}
\end{equation*}
where	
\begin{equation*}
S_k^N(t)=\int S\left(\sum_{|k|\leq N}F^N_k(t)e^{ik\cdot x}\right)e^{ik\cdot x}\,dx.
\end{equation*}
By Assumptions (A1), (A3) and the Cauchy-Lipschitz theorem we deduce that for any $N\in\N$ there exists $T^N>0$ and $(v^{N}, F^N)$ solution of \eqref{eq:app}-\eqref{eq:id2} of the form 
\begin{equation*}
\begin{aligned}
&v^N=\sum_{|k|\leq N}v^N_k(t)e^{ix\cdot k}\quad
&F^N=\sum_{|k|\leq N}F^N_k(t)e^{ix\cdot k}
\end{aligned}
\end{equation*}
and smooth on $(0,T^N)\times\T$. Next, we prove that $T^N=T$. By multiplying the first equation of \eqref{eq:app} by $v^N$, after integrating in space we get 
\begin{equation}\label{eq:a1}
\frac{d}{dt}\int\frac{|v^N|^2}{2}\,dx+\int P^N(S(F^N):\nabla\,v^N\,dx+\int|\nabla v^N|^2\,dx=0.
\end{equation}
Then, by multiplying the second equation of \eqref{eq:app} by $P^N(S(F^N))$ and using that 
$DW=S$ we have, after summing with \eqref{eq:a1}, that
\begin{equation}\label{eq:eeapp}
\frac{d}{dt}\int\left(\frac{|v^N|^2}{2}+W(F^N)\right)\,dx+\int|\nabla v^N|^2\,dx=0.
\end{equation}
By integrating in time, using (A2), (A3), the fact that $(v_0, F_0)\in L^{2}(\T)\times L^{p}(\T)$, and Parseval's identity we infer that 
\begin{equation*}
\sup_t\left(\sum_{|k|\leq N}|v^N_k(t)|^2+|F^N_k(t)|^2\right)\leq C
\end{equation*}
with $C>0$ not depending on $N$. Therefore, by standard O.D.E. theory it follows that 
$T^N=T$.  
\end{proof}

\subsection{Proof of Theorem \ref{teo:ex}}

The following lemma contains the main estimate of the paper, resulting in the propagation of $H^1$ regularity. We stress that the inequality \eqref{eq:h1estgal} in Lemma \ref{lem:h1} is a mere rephrasing of the {\em a priori } estimate \eqref{eq:unicab}. We also remark that constants depending on fixed parameters, e.g. the domain, dimension or fixed exponents, will be suppressed from appearing in inequalities. Instead, we adopt the notation $\lesssim$, meaning that all the terms on the right-hand side of $\lesssim$ are multiplied by constants depending only on the data except the ones where the constant is explicit. 
 
\begin{lemma}\label{lem:h1}
Assume that $W$ satisfies  (A1)-(A4) for some $p\geq 2$ and let $(v^N, F^N)$ be the solution of the Galerkin approximation constructed in Proposition \ref{prop:exgal}. Then, for any $t\in(0,T)$
\begin{equation}\label{eq:h1estgal}
\begin{aligned}
\int|\nabla\,F^N(t,x)|^2\,dx&\lesssim\int_0^t\int|\nabla v^N|^2\,dx\,ds+\sup_t\int|v^N(t,x)|^2\,dx\\
&\qquad +\int_0^t\int|\nabla F^N|^2\,dx\,ds+\int|v^N_0|^2\,dx+\int|\nabla F^N_0|^2\,dx. 
\end{aligned}
\end{equation}	
In addition, if $W$ satisfies (A4$^{\prime}$), it holds that 
\begin{equation}\label{eq:h1estgalp}
\begin{aligned}
\int|\nabla\,F^N(t,x)|^2\,dx&+\int_0^t\int|\nabla|F^N|^\frac{p}{2}|^2\,dx\,ds\lesssim\int_0^t\int|\nabla v^N|^2\,dx\,ds+\sup_t\int|v^N(t,x)|^2\,dx\\
&\qquad +\int_0^t\int|\nabla F^N|^2\,dx\,ds+\int|v^N_0|^2\,dx+\int|\nabla F^N_0|^2\,dx. 
\end{aligned}
\end{equation}	

\end{lemma}
\begin{proof}
By multiplying the first equation of \eqref{eq:app} by $-\dive F^N$ and integrating by parts we get
\begin{equation*}
-\int\,v^N_t\dive F^N\,dx+\int\dive(P^N(S(F^N))\dive F^N\,dx+\int\Delta\,v^N\dive\,F^N\,dx=0.
\end{equation*}
By using the second equation and the fact that $F^N$ is a gradient, after standard manipulations, it follows that 
\begin{equation*}
\frac{d}{dt}\int\left(\frac12 |\nabla\,F^N|^2-v^N\dive F^N\right)\,dx+ \int\dive(P^N(S(F^N))\dive F^N\,dx=\int|\nabla\,v^N|^2\,dx
\end{equation*}
and by using the definition of $\tilde{W}$ we have 
\begin{equation}\label{eq:h1est1.1}
\begin{aligned}
\frac{d}{dt}\int\left(\frac12 |\nabla\,F^N|^2-v^N\dive F^N\right)\,dx&+ \int\dive(P^N(\tilde{S}(F^N))\dive F^N\,dx=\int|\nabla\,v^N|^2\,dx\\
&+K\int|\dive\,F^N|^2\,dx.
\end{aligned}
\end{equation}
Regarding the third term on the left-hand side, by integrating by parts twice and using that $\curl\,F^N = 0$ we infer that
\begin{equation}\label{eq:convexity}
\int\dive(P^N(\tilde{S}(F^N))\dive F^N\,dx=\int\frac{\partial^2\tilde{W}(F^N)}{\partial F_{i\beta}^N\partial F_{j\gamma}^N}\partial_{\alpha} F_{i\beta}^N \partial_{\alpha} F_{j\gamma}^N\,dx\geq 0
\end{equation}
where $\tilde{S}=D\tilde{W}$ and we have used that $\tilde{W}$ is convex and therefore $D^{2}\tilde{W}\geq 0$.
Then, after integrating in time \eqref{eq:h1est1.1} and using \eqref{eq:convexity} we find that
\begin{equation*}
\begin{aligned}
\int|\nabla\,F^N|^2\,dx&\lesssim\int_0^t\int|\nabla\,v^N|^2\,dx\,ds+K\int_0^t\int|\dive\,F^N|^2\,dx\,ds\\
&\quad +\int|v^N||\dive F^N|\,dx
+\int|v^N_0||\dive F^N_0|\,dx+\int|\nabla\,F^N_0|^2\,dx.
\end{aligned}
\end{equation*}
Therefore, \eqref{eq:h1estgal} follows by a simple application of Young's inequality. Regarding \eqref{eq:h1estgalp} it is enough to note that, owing to assumption (A4$^{\prime}$), Lemma \ref{lem:modAB} implies that \eqref{eq:convexity} reads as 
\begin{equation*}
\int|\nabla|F^N|^\frac{p}{2}|^2\,dx\lesssim\int\frac{\partial^2\tilde{W}(F^N)}{\partial F_{i\beta}^N\partial F_{j\gamma}^N}\partial_{\alpha} F_{i\beta}^N \partial_{\alpha} F_{j\gamma}^N\,dx.
\end{equation*}
\end{proof}

We are now in a position to prove Theorem \ref{teo:ex}. 

\begin{proof}[Proof of Theorem \ref{teo:ex}]

We divide the proof in two steps.\\

\noindent{\em Step 1: Global existence for initial data $(v_0, F_0)$ in $L^2(\T)\times H^1(\T)$}.\\

We first prove the existence for initial data $(v_0, F_0)$ in $L^2(\T)\times H^1(\T)$. Note that by recalling hypothesis \eqref{eq:hypid}, equality \eqref{eq:eeapp} and Lemma \ref{lem:h1} we infer that 
\begin{equation*}
\begin{aligned}
&	\{v^N\}_N\subset L^{\infty}(0,T;L^{2}(\T))\cap L^{2}(0,T;H^{1}(\T))\\
&	\{F^N\}_N\subset L^{\infty}(0,T;H^1(\T))
\end{aligned}
\end{equation*}
with uniform bounds. Moreover, by using \eqref{eq:app} we also have that
\begin{equation*}
\begin{aligned}
&	\{\partial_t v^N\}_N\subset L^{2}(0,T;H^{-s}(\T)),\,\,s>\frac{d}{2}+1\\\
&	\{\partial_t F^N\}_N\subset L^{2}(0,T;L^2(\T)),
\end{aligned}
\end{equation*}
also with uniform bounds. Therefore, a simple application of the Aubin-Lions-Simon lemma \cite{Simon85} gives the existence of $(v,F)$ such that
\begin{equation*}
\begin{aligned}
&	v\in L^{\infty}(0,T;L^{2}(\T))\cap L^{2}(0,T;H^{1}(\T))\\
&	F\in L^{\infty}(0,T;H^1(\T))\cap C([0,T);L^{2}(\T)) 
\end{aligned}
\end{equation*}
and 
\begin{equation}\label{eq:conv_strong}
\begin{aligned}
&v^N\rightarrow v\mbox{ in }L^{2}(0,T;L^2(\T))\\
&F^N\rightarrow F\in C([0,T);L^r(\T))
\end{aligned}
\end{equation}
with $r<6$ if $d=3$ and $r<\infty$ if $d=2$.
Next, for any $\phi\in C^{\infty}_c([0,T);C^{\infty}(\T))$, it follows that 
\begin{equation}\label{eq:convtest}
P^{N}\phi\rightarrow\phi\mbox{ in }L^q((0,T)\times\T),\mbox{ for any }q<\infty,
\end{equation}
while, due to assumption (A3) and \eqref{eq:conv_strong}, we find that
\begin{equation}\label{eq:convstress}
S(F^N)\rightarrow S(F)\mbox{ in }L^{r} ((0,T)\times\T)\mbox{ for any }r<\frac{p}{p-1}.
\end{equation}
Then, combining \eqref{eq:conv_strong}-\eqref{eq:convstress} it is fairly straightforward to 
prove that $(v,F)$ is a weak solution of 
\eqref{eq:main}-\eqref{eq:id} in the sense of Definition \ref{def:ws}. Moreover, if $W$ satisfies assumption (A4$^{\prime}$), by Lemma \ref{lem:h1} we have that 
\begin{equation*}
\{\nabla|F^N|^{\frac{p}{2}}\}_N\subset L^{2}(0,T;L^{2}(\T)) 
\end{equation*}
which together with the bounds of $\{F^N\}$ in $L^\infty(0,T;H^1(\T))$ and the strong convergence in \eqref{eq:conv_strong} implies \eqref{eq:h1estp}.\\

\noindent{\em Step 2: Global existence for initial data $(v_0, F_0)\in L^{2}\times\,L^p(\T)$: Propagation of compactness.}\\

Next, we consider the case of initial data $(v_0, F_0)\in L^{2}\times\,L^p(\T)$. We note that since $p\geq 2$ given $F_0\in\,L^p(\T)$, there always exists a sequence of 
$\{F_0^n\}_n\subset\,H^{1}(\T)$ such that 
\begin{equation*}
\begin{aligned}
&F_0^n\rightarrow\,F_0\mbox{ in }L^{2}(\T)\\
&\sup_{n}\|F_0^n\|_p\leq C.
\end{aligned}	
\end{equation*}
From the previous part of the proof, we can claim the following: for any $n\in \mathbb{N}$, there exists $(v^n, F^n)$ weak solution of \eqref{eq:main} with initial data $(v_0, F_0^n)$ satisfying the following uniform bounds
\begin{equation}\label{eq:unifbd1}
\begin{aligned}
&	\{v^n\}_n\subset L^{\infty}(0,T;L^{2}(\T))\cap L^{2}(0,T;H^{1}(\T))\\
&	\{F^n\}_n\subset L^{\infty}(0,T;L^p(\T)).
\end{aligned}
\end{equation}
Moreover, we find that 
\begin{equation*}
\begin{aligned}
&	\{\partial_t v^n\}_n\subset L^{2}(0,T;W^{-1,\frac{p}{p-1}}(\T)),
\end{aligned}
\end{equation*}
also with uniform bounds. Therefore, proceeding as before, a simple application of the Aubin-Lions-Simon lemma gives the existence of $(v,F)$ such that
\begin{equation*}
\begin{aligned}
&	v\in L^{\infty}(0,T;L^{2}(\T))\cap L^{2}(0,T;H^{1}(\T))\\
&	F\in L^{\infty}(0,T;L^p(\T))
\end{aligned}
\end{equation*}
and 
\begin{equation}\label{eq:conv}
\begin{aligned}
&v^n\rightarrow v\mbox{ in }L^{2}(0,T;L^2(\T))\\
&F^n\weaktos F\in L^{\infty}(0,T;L^p(\T)).
\end{aligned}
\end{equation}
Next, note that 
\begin{equation*}
\{\partial_t F^n\}_n\subset L^{2}(0,T;L^2(\T))
\end{equation*}
and thus $F\in C([0,T];L^2(\T))$ and 
\begin{equation}\label{eq:td}
\frac{d}{dt}\int\,\frac{|F|^2}{2}\,dx=\int\,\partial_t\,F\,F\,dx.
\end{equation}
Indeed, this follows since $F$, $\partial_t F\in L^2(0,T;L^2(\T))$. Then, $F\in H^1(0,T;L^2(\T))$ which embeds into $C([0,T];L^2(\T))$ as $\|\int_s^\tau \partial_t F\|_2 \leq \int_s^\tau \|\partial_t F\|_2$.
Next, we note that by (A3) and the uniform bound of $\{F^n\}_n$ in 
$L^{\infty}(0,T;L^p(\T))$ we have that
\begin{equation}\label{eq:wsconv}
S(F_n)\weakto\overline{S(F)}\mbox{ in }L^{\frac{p}{p-1}}(\T)
\end{equation}
where $\overline{S(F)}$ denotes the weak limit of $S(F^n)$ which, at this moment, may not be equal to $S(F)$. In particular, testing with $\theta(t)\phi(t,x)$ where $\dot{\theta}$ is a function localising at time $t$, we infer that for a.a $t\in (0,T)$, $(v,F)$ satisfies 
\begin{equation}\label{eq:weaklim}
\iint\,v\partial_t\phi\,dx\,ds-\iint\,\overline{S(F)}\nabla\phi\,dx\,ds-\iint\nabla\,v\nabla\phi\,dx\,ds
-\int v(t)\phi(t)\,dx+\int\,v(0)\phi(0)\,dx=0
\end{equation}
for any $\phi\in C^{\infty}([0,T];C^{\infty}(\T))$.
Similarly, \eqref{eq:weaklim} holds for $(v^n, F^n)$. Then, subtracting the equations for $(v^n,F^n)$ from \eqref{eq:weaklim}, noting that $v^n_0 = v_0$, by a simple density argument, we may test with $\phi = y - y^n$ to infer that
\begin{equation}\label{eq:FDest0}
\begin{aligned}
\int\,\frac12 |F-F^n|^2\,dx+ & \int_0^t\int\,(\overline{S(F)}-S(F^n))(F-F^n)\,dx\,ds=\int_0^t\int|v-v^n|^2\,dx\,ds\\
& -\int(v-v^n)(y-y^n)\,dx+\int \frac12 |F_0-F^n_0|^2\,dx,
\end{aligned}
\end{equation}
where we have also made use of \eqref{eq:td}. For $\tilde{S}(F) = S(F) + KF$, \eqref{eq:FDest0} becomes
\begin{equation*}
\begin{aligned}
&\int\,\frac12 |F-F^n|^2\,dx+\int_0^t\int\,(\tilde{S}(F)-\tilde{S}(F^n))(F-F^n)\,dx\,ds=\int_0^t\int|v-v^n|^2\,dx\,ds\\
&-\int(v-v^n)(y-y^n)\,dx+\int \frac12 |F_0-F^n_0|^2\,dx+\int_0^t\int(S(F) - \overline{S(F)})(F-F^n)\,dx\,ds\\
&+ K\int_0^t\int|F-F^n|^2\,dx\,ds
\end{aligned}
\end{equation*}	
By assumption (A4) and \eqref{eq:unifbd1} we find that
\begin{equation*}
\begin{aligned}
\int\,|F-F^n|^2\,dx & \lesssim \int_0^t\int|v-v^n|^2\,dx\,ds+\left(\int|y-y^n|^2\,dx\right)^{\frac{1}{2}}\\
&\quad +\int|F_0-F^n_0|^2\,dx+\left|\int_0^t\int\,(S(F) - \overline{S(F)})(F-F^n)\,dx\,ds\right|\\
&\qquad +\int_0^t\int|F-F^n|^2\,dx\,ds,
\end{aligned}
\end{equation*}	
where the second term on the right-hand side is obtained via H\"older's inequality and the bound $v$, $v^n \in L^\infty(0,T;L^2(\T))$ in \eqref{eq:unifbd1}. Next, define
\begin{equation*}
\kappa(t):=\limsup_{n\to\infty}\int|F(t)-F^n(t)|^2\,dx.
\end{equation*}
Then, $\kappa(0)=0$, and note that 
\begin{equation*}
\begin{aligned}
&v^n\rightarrow v\mbox{ in }L^{2}(0,T;L^2(\T)),\\
&y^n\rightarrow y\mbox{ in }C(0,T;L^2(\T)).
\end{aligned}
\end{equation*}
Moreover, due to \eqref{eq:conv} and \eqref{eq:wsconv}, 
\begin{equation*}
\int_0^t\int\,(S(F) - \overline{S(F)})(F-F^n)\,dx\,ds\rightarrow 0\mbox{ as }n\to\infty
\end{equation*}
and we obtain that 
\begin{equation*}
\kappa(t)\leq \int_0^t\,\kappa(s)\,ds.
\end{equation*}
By Gr\"onwall's inequality we conclude that $\kappa(t)=0$ and therefore 
\begin{equation}\label{eq:strongf}
F_n\rightarrow\,F\mbox{ in } L^{2}(0,T;L^2(\T)).
\end{equation}
From \eqref{eq:strongf} we deduce that $\overline{S(F)}=S(F)$ and thus $(v,F)$ is a weak solution of \eqref{eq:main}-\eqref{eq:id}.
\end{proof}

\subsection{Proof of Theorem \ref{teo:energy}}

We prove the theorem concerning the conservation of energy. 

\begin{proof}[Proof of Theorem \ref{teo:energy}]
Let $(v,F)$ be a weak solution of the system \eqref{eq:main}-\eqref{eq:id} in the sense of Definition \ref{def:ws} such that 
\begin{equation}\label{eq:cons1}
F\in L^{\infty}(0,T;H^{1}(\T)). 
\end{equation}
By Sobolev embedding we get 
\begin{equation*}
\begin{aligned}
& F\in L^{\infty}(0,T;L^r(\T)),\mbox{ for }d=2\mbox{ and for all }r<\infty,\\
& F\in L^{\infty}(0,T;L^6(\T))\mbox{ for }d=3.
\end{aligned}
\end{equation*}
Then by using the growth condition of $S$ in (A2) we have that 
\begin{equation}\label{eq:cons3}
S(F)\in L^{\infty}(0,T;L^{2}(\T))
\end{equation}
for any $p\geq 2$ in the case $d=2$ and for $2\leq p\leq 4$ in the case $d=3$.\par
Next, the weak formulation \eqref{eq:ws1}, \eqref{eq:ws2} and \eqref{eq:cons3} imply that
\begin{align}
&\partial_t\,F=\nabla v,\mbox{ a.e. on $(0,T)\times\T$}\label{eq:cons4}\\
&\partial_t v=\dive(S(F)+\nabla u)\mbox{ in }L^{2}(0,T;H^{-1}(\T)),\label{eq:cons5}
\end{align}
where $H^{-1}(\T):=(H^{1}(\T))^{\prime}$. First, we note that regarding the time chain-rule for $W(F)$, by using the bounds \eqref{eq:cons4}, \eqref{eq:cons1} and  \eqref{eq:ws1}, and by using the weak formulation \eqref{eq:ws2}, we can infer that $F$ can be re-defined on a set of measure zero in time so that $F\in C([0,T];L^{2}(\T))$ and 
\begin{equation*}
\begin{aligned}
&\lim_{t\to 0}F(t)=F_0\mbox{ strongly in }L^2(\T),\\
&\lim_{t\to 0}F(t)=F_0\mbox{ weakly in }H^{1}(\T).
\end{aligned}
\end{equation*}
Then, assumption (A2) on the growth of $S(F)$ and the fact that $F\in C([0,T];L^{2}(\T))$ imply that 
$W(F(\cdot))\in C([0,T];L^{1}(\T))$ and $W(F(t))\to W(F_0)$ in $L^{1}(\T)$ as $t\to 0$. Finally, again by using \eqref{eq:cons4}, \eqref{eq:ws1} and \eqref{eq:cons3} we have that $\partial_t W(F)=S(F):\partial_t F$ a.e. in $(0,T)\times\T$ and then  $\partial_t W(F)\in L^1(0,T;L^{1}(\T))$. Then, by standard properties of the Bochner integral we have that 
\begin{equation}\label{eq:cons6}
\int W(F(t))\,dx-\int W(F_0)\,dx=\int\,S(F):\nabla v.
\end{equation}
Next, regarding the velocity, we have that  \eqref{eq:ws2}, \eqref{eq:ws3}, \eqref{eq:cons5}, and \eqref{eq:cons3} imply that $\partial_tv\in L^{2}(0,T;H^{-1}(\T))$ and therefore,  by the Lions-Magenes Lemma, after a possible re-definition on a set of measure zero in time, we  have that $v\in C([0,T];L^{2}(\T))$ and then from \eqref{eq:ws2} it follows that
\begin{equation*}
\lim_{t\to 0}v(t)=v_0\mbox{ strongly in }L^{2}(\T). 
\end{equation*}
Moreover, again from the Lions-Magenes Lemma, we have that 
\begin{equation*}
\begin{aligned}
&\frac{d}{dt}\int\,\frac{|v|^2}{2}\,dx= \langle\partial_t v, v\rangle_{H^{-1} , H^{1}}
\end{aligned}
\end{equation*}
and then 
\begin{equation}\label{eq:cons7}
\int\frac{|v(t)|^2}{2}\,dx-\int\frac{|v_0|^2}{2}\,dx=-\int\left(S(F)+\nabla v\right):\nabla v\,dx.
\end{equation}
Combining \eqref{eq:cons6} and \eqref{eq:cons7} completes the proof.  
\end{proof}

\section{Uniqueness in two dimensions}\label{sec5}

In this section, we restrict to two space dimensions and prove Theorem \ref{teo:uni} and Theorem \ref{teo:unip}. 

\begin{proof}[Proof of Theorem \ref{teo:uni}]

Given the initial data $(v_0,F_0)$ satisfying the hypothesis of Theorem \ref{teo:uni}, the existence follows from the Theorem \ref{teo:ex}. Regarding uniqueness, let $(v_1,F_1)$ and $(v_2,F_2)$ be two weak solutions \eqref{eq:main}-\eqref{eq:id} and note that from Theorem \ref{teo:ex} we have that  
\begin{equation}\label{eq:61}
\begin{aligned}
& v\in L^{\infty}(0,T;L^2(\Tt))\cap L^{2}(0,T;H^{1}(\Tt)),\\
& F\in C(0,T;L^{r}(\Tt))\cap L^{\infty}(0,T;H^{1}(\Tt)),\mbox{ for any }r<\infty.
\end{aligned}
\end{equation}
Since $\nabla\,v\in L^{2}(0,T;L^2(\Tt))$ the second equation in \eqref{eq:main} is satisfied almost everywhere on $(0,T)\times\Tt$ and, therefore, by an approximation argument and \eqref{eq:61}, we deduce that for any $1<r<\infty$ and any $t\in (0,T)$
\begin{equation}\label{eq:62}
\int\,|F_1-F_2|^r\,dx\lesssim \int_0^t\int|\nabla\,v_1-\nabla\,v_2| |F_1-F_2|^{r-1}\,dx\,ds
\end{equation}
Next, by setting $w=v_1-v_2$ we infer that
\begin{align*}
\partial_t w-\Delta w&=\dive(S(F_1)-S(F_2))\\
w|_{t=0}&=0
\end{align*}
weakly on $(0,T)\times\Tt$. Then, given $1<r<\infty$, by Lemma \ref{lem:maxlp}, we find that for any $t\in(0,T)$
\begin{equation}\label{eq:63}
\int_0^t\|\nabla\,v_1-\nabla\,v_2\|_{r}^{r}\,ds\lesssim \int_0^t \|S(F_1)-S(F_2)\|_{r}^{r}\,ds.
\end{equation}
In particular, by using Young's inequality and combining \eqref{eq:62} and \eqref{eq:63}, we get that 
\begin{equation}\label{eq:64}
\begin{aligned}
\int\,|F_1-F_2|^r\,dx&\lesssim \int_0^t\int|F_1-F_2|^{r}\,dx\,ds+\int_0^t\int|S(F_1)-S(F_2)|^{r}\,dx\,ds.
\end{aligned}
\end{equation}
Next, as a consequence of \eqref{eq:growthsecond}, we have that
\begin{equation*}
|S(F_1)-S(F_2)| \lesssim (1+|F_1|^{s}+|F_2|^{s})|F_1-F_2|
\end{equation*}
and then \eqref{eq:64} becomes
\begin{align}
\int\,|F_1-F_2|^r\,dx&\lesssim \int_0^t\int|F_1-F_2|^{r}\,dx\,ds\nonumber\\
&\quad +\int_0^t\int|F_1|^{rs}|F_1-F_2|^{r}\,dx\,ds\nonumber\\
&\quad +\int_0^t\int|F_2|^{rs}|F_1-F_2|^{r}\,dx\,ds\nonumber\\
&\qquad =I_1+I_2+I_3. \label{eq:65}
\end{align}

We now estimate the term $I_2$.  Let $q>1$ and since  
$\tfrac{1}{r q}+ \tfrac{q-1}{r q} + \tfrac{r-1}{r} = 1$ H\"older's inequality implies
\begin{align*}
I_2&\leq \int_0^t\int|F_1|^{rs}|F_1-F_2| |F_1-F_2|^{r-1}\,ds\\
   &\leq \int_0^t\|F_1\|_{r^2qs}^{rs}\| F_1-F_2\|_{\frac{rq}{q-1}}\|F_1-F_2\|_{r}^{r-1}\,ds \, .
  \end{align*}
Note that since $0\leq \gamma\leq 2$, $2\leq p<2+\gamma$, and $p-2\leq s\leq p-\gamma$ we have that $\frac{2}{s}>1$ and then by choosing $r=\frac{2}{s}$ and by using consecutively  Lemma \ref{lem:gani} (i) and (ii), we infer that  
\begin{align*}
I_2 &\leq \int_0^t\|F_1\|_{\frac{4q}{s}}^{2}\| F_1-F_2\|_{\frac{rq}{q-1}}\|F_1-F_2\|_{r}^{r-1}\,ds\\
   &\lesssim q\int_0^t\|\nabla\,F_1\|^{2}_2\|F_1-F_2\|_{r}^{1-\frac{1}{q}}
   \|\nabla\,F_1-\nabla\,F_2\|_{2}^{\frac{1}{q}} \|F_1-F_2\|_{r}^{r-1}\,ds,
\end{align*}
where, the suppressed constant depends on $s$, $p$ but not $q$. Then, since $F_1,F_2\in L^{\infty}(0,T;H^{1}(\Tt))$, we deduce that
\begin{equation*}
I_2\lesssim q \int_0^t\,\|F_1-F_2\|_{r}^{r-\frac{1}{q}}\,ds.
\end{equation*}
Arguing in the exact same way for $I_3$ we get from \eqref{eq:65} that
\begin{align*}
\|F_1-F_2\|_{r}^r&\lesssim \int_0^t \|F_1-F_2\|_{r}^r\,ds+q \int_0^t\|F_1-F_2\|_{r}^{r-\frac{1}{q}}\,ds\nonumber\\
&\lesssim q\, (\sup_t\|F_1 - F_2\|^{\frac1q}_{r} + 1)  \int_0^t\|F_1-F_2\|_{r}^{r-\frac{1}{q}}\,ds\nonumber\\
&\lesssim q \int_0^t\|F_1-F_2\|_{r}^{r-\frac{1}{q}}\,ds
\end{align*}
where we have used again that $F_1,F_2\in L^{\infty}(0,T;H^{1}(\Tt))$.

Next, set $y(t)= \|F_1(t)-F_2(t)\|_{r}^r$ and notice that $y(t)\in C([0,T])$; then, for any $q>1$ it holds that 
\begin{equation*}
y(t)\leq Crq\int_0^t\,y(s)^{1-\frac{1}{rq}}\,ds,\,\,y(0)=0,
\end{equation*}
where we recall that $r=\frac{2}{s}$ is fixed and $C$ does not depend on $q$. Then, the fact that $y(t)=0$ on $(0,T)$ follows from Lemma \ref{lem:bihari}, which implies that $0\leq y(t)\leq \bar{y}(t)=(Ct)^{rq}$ with $C$ independent of $q$. In particular, since $\bar y$ is a non-decreasing function, we may choose $t^*\in (0,T)$ such that for all $q>1$
\[
y(t) \leq \left(\frac12\right)^q,\,\,\forall\,t\in(0,t^*).
\]
Taking $q\to\infty$ we deduce that $y\equiv 0$ on $(0,t^*)$ and, by repeating the argument, that $y\equiv 0$ on $(0,T)$.

Therefore, $F_1=F_2$ and $v_1-v_2$ solves in the weak sense 
\begin{equation*}
\begin{aligned}
\partial_t(v_1-v_2)-\Delta(v_1-v_2)&=0\\
(v_1-v_2)|_{t=0}&=0.
\end{aligned}
\end{equation*}
Hence, $v_1=v_2$ due to standard uniqueness results for the heat equation.
\end{proof}

We next proceed with the proof of Theorem \ref{teo:unip}. This follows by an argument similar to Theorem \ref{teo:uni} and we only show the few modifications required.

\begin{proof}[Proof of Theorem \ref{teo:unip}]
As in Theorem \ref{teo:uni}, given $(v_0,F_0)$ as in the statement of Theorem \ref{teo:unip} the existence part follows by Theorem \ref{teo:ex}. Regarding uniqueness, arguing exactly as in the proof of Theorem \ref{teo:uni}
we get \eqref{eq:65}, which we rewrite for the reader's convenience 
\begin{equation*}
\begin{aligned}
\int\,|F_1-F_2|^r\,dx&\lesssim \int_0^t\int|F_1-F_2|^{r}\,dx\,ds\\
&\quad +\int_0^t\int|F_1|^{rs}|F_1-F_2|^{r}\,dx\,ds\\
&\qquad +\int_0^t\int|F_2|^{rs}|F_1-F_2|^{r}\,dx\,ds\\
&\qquad =I_1+I_2+I_3.
\end{aligned}
\end{equation*}
To estimate the term $I_2$, set $U_i:=|F_i|^{\frac{p}{2}}$, so that

\begin{equation*}
\begin{aligned}
I_2
&= \int_0^t \int U_1^{\frac{2rs}{p}} |F_1-F_2| |F_1-F_2|^{r-1}\,dx\,ds.
\end{aligned}
\end{equation*}
Note that $p\geq2$ and the hypothesis $s<p$ implies that $r:=\frac{p}{s}\in (1,\infty)$. Then, as in the proof of Theorem \ref{teo:uni}, by using Lemma \ref{lem:gani} we now get that 

\begin{equation*}
\begin{aligned}
I_2 &\leq \int_0^t\|U_1\|_{2rq}^{2}\| F_1-F_2\|_{\frac{rq}{q-1}}\|F_1-F_2\|_{r}^{r-1}\,ds\\
   &\lesssim q\int_0^t\|U_1\|^{2}_{H^1}\|F_1-F_2\|_{r}^{1-\frac{1}{q}}
   \|\nabla\,F_1-\nabla\,F_2\|_{2}^{\frac{1}{q}} \|F_1-F_2\|_{r}^{r-1}\,ds,
\end{aligned}
\end{equation*}
where the suppressed constant depends on $r$, $p$ but not $q$. Since $F_1,F_2\in L^{\infty}(0,T;H^{1}(\Tt))$
\begin{equation*}
I_2\lesssim q  \int_0^t\,\|U_1(s)\|_{H^1}^2\|F_1(s)-F_2(s)\|_r^{r-\frac{1}{q}}\,ds.
\end{equation*}
We may treat $I_3$ similarly to get that $y(t)= \|F_1(t)-F_2(t)\|_{r}^r\in C([0,T])$ satisfies
\begin{equation}\label{eq:finaluni}
y(t)\leq C rq\int_0^t\,f(s)y(s)^{1-\frac{1}{rq}}\,ds,\,\,y(0)=0,
\end{equation}
where $f(s) := 1+\|U_1(s)\|_{H^1}^2 + \|U_2(s)\|_{H^1}^2 \in L^1(0,T)$, $C$ is independent of $q$ and \eqref{eq:finaluni} holds for all $q>1$. Then, by using again Lemma \ref{lem:bihari} we deduce that $0 \leq y(t) \leq \bar y(t) = \left(C\int_0^t f(s) \,ds\right)^q$ and the proof can now be concluded as in Theorem \ref{teo:uni}.
\end{proof}

\section{Global regularity}\label{sec6}

Before proving Theorem \ref{teo:reg} we prove a local existence result for \eqref{eq:main}-\eqref{eq:id}. Of course, the local existence of smooth solutions holds under more general hypothesis than the ones considered in the next proposition.
\begin{proposition}\label{prop:local}
Under the assumptions of Theorem \ref{teo:reg}, there exists a time 
$T^*=T^*(\|v_0\|_{H^{3}}, \|F_0\|_{H^{3}} )$ such that the unique weak solution constructed in Theorem \ref{teo:uni} satisfies 
\begin{equation*}
(v,F)\in L^{\infty}(0,T^*;H^3(\Tt))
\end{equation*}
\end{proposition}
\begin{proof}
We henceforth use the following notation: given tensors $\{A^i\}_{i=1}^n$, we write
\[
(A^1)(A^2)....(A^n)
\]
for linear combinations of products of the entries of the tensor inside the brackets. Recall the Galerkin approximation
\begin{equation}\label{eq:app1}
\begin{aligned}
\partial_t\,v^N-\dive\,P^N(S(F^N))-\Delta\,v^N&=0\\
\partial_t\,F^N-\nabla\,v^N&=0\\
\curl\,F^N&=0,
\end{aligned}
\end{equation}
and note that, in the periodic setting, the operators $P^N$ and $\nabla$ commute. Therefore, 
\begin{equation*}
\nabla^2\dive\,P^N(S(F^N))=P^N \nabla^2\dive\, S(F^N)
\end{equation*}
and 
\begin{equation*}
\|P^N \nabla^2\dive\, S(F^N)\|_2\leq C\|\nabla^3\,S(F^N)\|_2
\end{equation*}
with $C$ independent of $N$.
Next, note that 
\begin{equation*}
\begin{aligned}	
\nabla^3(S(F^N))&=(D^4W(F^N))(\nabla\, F^N)(\nabla\, F^N)(\nabla\, F^N)+(D^3W(F^N))(\nabla^2\, F^N)(\nabla\, F^N)\\
&+(D^2W(F^N))(\nabla^3\, F^N)
\end{aligned}
\end{equation*}
and then since all the derivative of $W$ have polynomial growth, by Sobolev embeddings it follows that there exists $m>0$ such that 
\begin{equation*}
\|\nabla^3\,S(F^N)\|_2^2\leq\,C \|\nabla^3\,F^N\|_2^{2+2m}
\end{equation*}
again with $C$ independent of $N$. By differentiating two times the first equation of \eqref{eq:app1} and multiplying the resulting equation by $-\Delta\nabla^2\,v^N$, after integrating in space, we obtain
\begin{equation}\label{eq:p5}
\begin{aligned}
\frac{d}{dt}\|\nabla^3\,v^N\|_2^2+\|\nabla^4\,v^N\|_2^2& \lesssim \|\nabla^2\dive(S(F^N))\|_2^2\\
&\lesssim \|\nabla^3\,F^N\|_2^{2+2m}.
\end{aligned}	
\end{equation}
On the other hand by differentiating three times the second equation and multiplying the resulting equation by $\nabla^3\,F^N$, integrating in space and using Young's inequality we get
\begin{equation}\label{eq:p6}
\frac{d}{dt}\|\nabla^3\,F^N\|_2^2 \lesssim \|\nabla^3\,F^N\|_2^2 + \frac{1}{4}\|\nabla^4\,v^N\|_2^2,
\end{equation}
where we note that no constant has been suppressed for the term $\frac{1}{4}\|\nabla^4\,v^N\|_2^2$. Therefore, by \eqref{eq:p5} and \eqref{eq:p6}, $y(t):=(\|\nabla^3\,F^N(t)\|_2^2+ \|\nabla^3\,v^N(t)\|_2^2)$ satisfies 
\begin{equation*}
\dot{y}(t)\lesssim y(t)+y^{1+m}(t)
\end{equation*}
which implies that there exists $T^*$ independent of $N$ such that 
\begin{equation*}
(v^N,F^N)\in L^{\infty}(0,T^*;H^{3}(\Tt))\times L^{\infty}(0,T^*;H^{3}(\Tt)).
\end{equation*}	
The proposition now follows by taking the limit as $N\to\infty$ to obtain 
\[
(v,F) \in L^{\infty}(0,T^*;H^{3}(\Tt))\times L^{\infty}(0,T^*;H^{3}(\Tt))
\]
which is the solution constructed in Theorem \ref{teo:ex} whose uniqueness follows from Theorem \ref{teo:uni}.
\end{proof}
We are now in a position to prove Theorem \ref{teo:reg}.
\begin{proof}[Proof of Theorem \ref{teo:reg}]
We only prove the \emph{a priori} estimates. Indeed, by using Proposition \ref{prop:local} the result follows by a simple continuity argument. Note that we have already proved 
\begin{equation*}
\begin{aligned}
&v\in L^{\infty}(0,T;L^{2}(\Tt))\cap L^{2}(0,T;H^{1}(\Tt))\\
& F\in L^{\infty}(0,T;H^{1}(\Tt)).
\end{aligned}
\end{equation*}
Next, we prove an $H^2$-estimate. By multiplying the first equation of \eqref{eq:main} by $-\Delta\,v$ we get 
\begin{equation*}
\frac{d}{dt}\|\nabla\,v\|_2^2+\|\nabla^2\,v\|_2^2\lesssim \int|D^2\tilde{W}(F)||\nabla\,F||\nabla^2\,v|\,dx+ K\int |\nabla\,F||\nabla^2\,v|\,dx
\end{equation*}
and integrating in time on $(0,t)$, suppressing $K$, we get that
\begin{equation}\label{eq:711.1}
\begin{aligned}
\int|\nabla\,v|^2\,dx+\int_0^t\int|\nabla^2\,v|^2\,dx\,ds\lesssim \int|\nabla\,v_0|^2\,dx&+ \int_0^t\int |D^2\tilde{W}(F)| |\nabla\,F||\nabla^2\,v|\,dx\,ds\\&+\int_0^t\int|\nabla\,F||\nabla^2\,v|\,dx\,ds.
\end{aligned}
\end{equation}
Next we multiply the first equation of\eqref{eq:main} by $\Delta\dive\,F$ and we integrate in space to get that
\begin{equation*}
\int v_t \Delta\dive\,F\,dx-\int\dive (S(F)) \Delta\dive\,F\,dx-\int\Delta\,v \Delta\dive\,F\,dx=0
\end{equation*}
and we treat every term separately. We start with the third term. By using the second and the third equation in \eqref{eq:main} we have
\begin{equation}\label{eq:73}
\begin{aligned}
-\int\Delta\,v \Delta\dive\,F\,dx&=-\int\dive\,F_t \Delta\dive\,F\,dx = \frac{d}{dt}\int\frac12|\nabla^2\,F|^2\,dx.
\end{aligned}
\end{equation}
Concerning the second term and using that fact that $\curl\,F = 0$ we find that
\begin{align}
-\int\partial_{\alpha} S_{i\alpha}(F) \partial_{\gamma\gamma} \partial_{\beta} F_{i\beta}
\,dx&= -\int\partial_{\beta} S_{i\alpha}(F) \partial_{\gamma\gamma} \partial_{\alpha} F_{i\beta}\,dx\nonumber\\
= -\int\frac{\partial^2\,W(F)}{\partial\,F_{i\alpha}\partial\,F_{j\delta}}
\partial_{\beta}F_{j\delta}\partial_{\gamma\gamma}\partial_{\beta}F_{i\alpha}\,dx
&=\int\frac{\partial^2\,W(F)}{\partial\,F_{i\alpha}\partial\,F_{j\delta}}
\partial_{\beta\gamma}F_{j\delta}\partial_{\gamma\beta}F_{i\alpha}\,dx\nonumber\\
+\int(D^3\,W(F))(\nabla\,F)(\nabla\,F)(\nabla^2\,F)\,dx&=\int\frac{\partial^2\,\tilde{W}(F)}{\partial\,F_{i\alpha}\partial\,F_{j\delta}}
\partial_{\beta\gamma}F_{j\delta}\partial_{\gamma\beta}F_{i\alpha}\,dx\nonumber\\
-K\int\,|\nabla^2\,F|^2\,dx&+\int(D^3\,W(F))(\nabla\,F)(\nabla\,F)(\nabla^2\,F)\,dx.\label{eq:74}
\end{align}	
Finally, concerning the first term we have 
\begin{equation}\label{eq:75}
\begin{aligned}
\int\,v_t\Delta\dive\,F\,dx&=\frac{d}{dt}\int v \Delta\dive\,F\,dx-\int|\nabla^2v|^2\,dx
\end{aligned}
\end{equation}
and combining \eqref{eq:73}-\eqref{eq:75} we infer that 
\begin{equation}\label{eq:76}
\begin{aligned}
&\frac{d}{dt}\int\left(|\nabla^2\,F|^2+ v \Delta\dive\,F\right)\,dx+\int\frac{\partial^2\,\tilde{W}(F)}{\partial\,F_{i\alpha}\partial\,F_{j\delta}}\partial_{\beta\gamma}F_{j\delta}\partial_{\gamma\beta}F_{i\alpha}\,dx\\
&\lesssim \int\,|\nabla^2\,F|^2\,dx
+\int|\nabla^2v|^2\,dx+\int(D^3\,W(F))(\nabla\,F)(\nabla\,F)(\nabla^2\,F)\,dx.
\end{aligned}
\end{equation}
Integrating \eqref{eq:76} in time we get 
\begin{align}
&\int|\nabla^2\,F|^2\,dx+\int_0^t\int\frac{\partial^2\,\tilde{W}(F)}{\partial\,F_{i\alpha}\partial\,F_{j\delta}}\partial_{\beta\gamma}F_{j\delta}\partial_{\gamma\beta}F_{i\alpha}\,dx\,ds\nonumber\\
&\lesssim \int_0^t\int\,|\nabla^2\,F|^2\,dx\,ds
+\int_0^t\int|\nabla^2v|^2\,dx\,ds+\int_0^t\int|D^3\,W(F)||\nabla\,F|^2|\nabla^2\,F|\,dx\,ds\nonumber\\
&\quad +\int\left(|\nabla^2\,F_0|^2+ |\nabla v_0|^2 +|\Delta F_0|^2\right)\,dx+\left|\int \nabla v \Delta\,F\,dx\right|\label{eq:761.1}\\
&\lesssim \int_0^t\int\,|\nabla^2\,F|^2\,dx\,ds
+\int_0^t\int|\nabla^2v|^2\,dx\,ds+\int_0^t\int|D^3\,W(F)||\nabla\,F|^2|\nabla^2\,F|\,dx\,ds\nonumber\\
&\quad +\int\left(|\nabla^2\,F_0|^2+ |\nabla v_0|^2 +|\Delta F_0|^2\right)\,dx+\int |\nabla v|^2\,dx+\frac{1}{2} \int|\nabla^2\,F|^2\,dx,\nonumber
\end{align}
where, by Young's inequality, there is no suppressed constant in front of the term $\frac{1}{2} \int|\nabla^2\,F|^2\,dx$. Also, note that by the convexity of $\tilde{W}$, (A4), the second term on the left-hand side is nonnegative and thus \eqref{eq:761.1} now reads as
\begin{equation}\label{eq:761}
\begin{aligned}
\int|\nabla^2\,F|^2 \,dx& \lesssim \int_0^t\int\,|\nabla^2\,F|^2\,dx\,ds
+\int_0^t\int|\nabla^2v|^2\,dx\,ds\\
&+\int_0^t\int|D^3\,W(F)||\nabla\,F|^2|\nabla^2\,F|\,dx\,ds \\
&\quad +\int\left(|\nabla^2\,F_0|^2+ |\nabla v_0|^2 \right)\,dx+\int |\nabla v|^2\,dx.
\end{aligned}
\end{equation}
We may now multiply \eqref{eq:761} by a constant small enough so that, after adding up with \eqref{eq:711.1},
\begin{equation}\label{eq:77}
\begin{aligned}
\int\left(|\nabla^2\,F|^2+|\nabla\,v|^2\right)\,dx+\int_0^t\int|\nabla^2\,v|^2\,dx\,ds&\lesssim
\int_0^t\int|\nabla^2\,F|^2\,dx\,ds\\&+\int_0^t\int|D^3W(F)||\nabla\,F|^2|\nabla^2\,F|\,dx\,ds\\
&+\int_0^t\int|D^2\tilde{W}(F)||\nabla\,F||\nabla^2\,v|\,dx\,ds\\& + \int_0^t\int |\nabla\,F||\nabla^2\,v|\,dx\,ds \\
&+\int\left(|\nabla^2\,F_0|^2+|\nabla\,v_0|^2\right)\,dx.
\end{aligned}
\end{equation}
We use now the growth conditions \eqref{eq:addhyp}. We only treat the case $p>3$, which is the case when $D^3\tilde{W}$ is not bounded. Thus, in the case we assume $(A4)$ we restrict out attention to the range $3< p\leq 5$ and from \eqref{eq:77} we get
\begin{equation}\label{eq:78}
\begin{aligned}
\int\left(|\nabla^2\,F|^2+|\nabla\,v|^2\right)\,dx+\int_0^t\int|\nabla^2\,v|^2\,dx\,ds&\lesssim
\int_0^t\int|\nabla^2\,F|^2\,dx\,ds\\
&+\int_0^t\int (1+ |F|^{p-3})|\nabla\,F|^2|\nabla^2\,F|\,dx\,ds\\
&+\int_0^t\int (1+ |F|^{p-2})|\nabla\,F||\nabla^2\,v|\,dx\,ds\\
&+\int\left(|\nabla^2\,F_0|^2+|\nabla\,v_0|^2\right)\,dx\\
&=I_1+I_2+I_3+I_4.
\end{aligned}
\end{equation}
We only have to bound $I_2$ and $I_3$. By using Young's inequality, Lemma \ref{lem:gani} (1) and the bound $F\in L^{\infty}(0,T;H^{1}(\Tt))$ we have 
\begin{equation*}
\begin{aligned}
I_3&\lesssim \int_0^t \left(1 + \|F\|_{4(p-2)}^{p-2}\right)\|\nabla\,F\|_4\|\nabla^2\,v\|_2\,ds\\
&\lesssim \int_0^t \left(1 + \|F\|_{4(p-2)}^{2(p-2)}\right)\|\nabla\,F\|_2\|\nabla^2\,F\|_2\,ds
+\frac{1}{4}\int_0^t \|\nabla^2\,v\|_2^2\,ds\\
&\lesssim \int_0^t\|\nabla^2\,F\|_2 \,ds+\frac{1}{4}\int_0^t \|\nabla^2\,v\|_2^2\,ds,
\end{aligned}
\end{equation*}
where, as in other instances, due to Young's inequality, there is no suppressed constant in front of $\frac{1}{4}\int_0^t \|\nabla^2\,v\|_2^2\,ds$ and can get absorbed into the left-hand side of \eqref{eq:78}. 

Regarding $I_2$, we proceed similarly and using Lemma \ref{lem:gani} (1), and the Brezis-Gallouet inequality (3) together with the bound $F\in L^{\infty}(0,T;H^{1}(\Tt))$, we find that 
\begin{equation*}
\begin{aligned}
I_2&\lesssim \int_0^t \left(1 + \|F\|_{\infty}^{p-3}\right)\|\nabla\,F\|_4^2\|\nabla^2\,F\|_2\,ds\\
&\lesssim \int_0^t \left(1 + (\log(e+\|\nabla^2\,F\|_2^2))^{\frac{p-3}{2}}\right)\|\nabla^2\,F\|_2^2\,ds\\
&\lesssim \int_0^t(\log(e+\|\nabla^2\,F\|_2^2))^{\frac{p-3}{2}}\|\nabla^2\,F\|_2^2\,ds\\
&\lesssim \int_0^t(\log(e+\|\nabla^2\,F\|^2_2))\|\nabla^2\,F\|_2^2\,ds
\end{aligned}
\end{equation*}
where we used that $3<p\leq5$, i.e. that $(p-3)/2 \in (0,1]$. We note that for $2\leq p\leq 3$ there is no need to use the Brezis-Gallouet inequality since $|D^3W|\leq C$. Therefore $y(t):=(\|\nabla^2\,F(t)\|_2^2+\|\nabla\,v(t)\|_2^2)$ satisfies the following differential inequality
\begin{equation*}
y(t)\lesssim y(0)+\int_0^t(y(s)+ y^{\frac{1}{2}}(s)+y(s)\log(e+y(s)))\,ds
\end{equation*}
which implies that 
\begin{equation}\label{eq:711}
\begin{aligned}
& v\in L^{\infty}(0,T;H^1(\Tt)\cap L^{2}(0,T;H^{2}(\Tt)),\\
& F\in L^{\infty}(0,T;H^{2}(\Tt)).
\end{aligned}
\end{equation}
In particular, from \eqref{eq:711}, it follows that 
\begin{equation}\label{eq:712}
\begin{aligned}
& F\in L^{\infty}((0,T)\times\Tt),\\
& \nabla\,F\in L^{\infty}(0,T;L^q(\Tt))\mbox{ for any }q<\infty.
\end{aligned}
\end{equation}
It remains to prove the $H^3$-estimate. As in the proof of Proposition \ref{prop:local} note that 
\begin{equation*}
\begin{aligned}	
\nabla^3(S(F))&=(D^4W(F))(\nabla\,F)(\nabla\,F)(\nabla\,F)+(D^3W(F))(\nabla^2\,F)(\nabla\,F)+(D^2W(F))(\nabla^3\,F).
\end{aligned}
\end{equation*}
Therefore, using that $W$ and its derivative have polynomial growth, and Lemma \ref{lem:gani} (1), we get from \eqref{eq:711} and \eqref{eq:712} that
\begin{equation*}
\|\nabla^3\,S(F)\|_2^2\lesssim 1+\|\nabla^3\,F\|_2+ \|\nabla^3\,F\|_2^2.
\end{equation*}
Arguing as in Proposition \ref{prop:local} we get the following two differential inequalities in analogy to \eqref{eq:p5}, \eqref{eq:p6}:
\begin{equation}\label{eq:714}
\begin{aligned}
\frac{d}{dt}\|\nabla^3\,v\|_2^2+\|\nabla^4\,v\|_2^2&\lesssim \|\nabla^2\dive(S(F))\|_2^2\\
&\lesssim 1+\|\nabla^3\,F\|_2+ \|\nabla^3\,F\|_2^2
\end{aligned}	
\end{equation}
and
\begin{equation}\label{eq:715}
\frac{d}{dt}\|\nabla^3\,F\|_2^2\lesssim \|\nabla^3\,F\|_2^2+\frac{1}{4}\|\nabla^4\,v\|_2^2,
\end{equation}
where there are no suppressed constants in the term $\frac{1}{4}\|\nabla^4\,v\|_2^2$. Therefore from \eqref{eq:714} and \eqref{eq:715}, $y(t):=(\|\nabla^3\,F(t)\|_2^2+ \|\nabla^3\,v(t)\|_2^2)$ satisfies 
\begin{equation*}
\dot{y}(t)\leq 1+y^{\frac{1}{2}}(t)+y(t)
\end{equation*}
which implies that
\begin{equation*}
(v,F)\in L^{\infty}(0,T;H^{3}(\Tt))\times L^{\infty}(0,T;H^{3}(\Tt)).
\end{equation*}	
Then the case of $W$ satisfying (A4) and $2\leq p\leq 5$ follows by a simple continuation argument.\par
Regarding the case of $W$ satisfying (A4$^\prime$) and $2\leq p\leq 6$ the proof is exactly the same except for the way we treat the term $I_2$ in the inequality \eqref{eq:78}. We first notice that by exploiting the assumption  (A4$^\prime$) in the inequality \eqref{eq:761.1} we have that \eqref{eq:78} reads as follows 
\begin{equation*}
\begin{aligned}
&\int\left(|\nabla^2\,F|^2+|\nabla\,v|^2\right)\,dx+\int_0^t\int|\nabla^2\,v|^2+\int_0^t\int|F|^{p-2}|\nabla^2F|^2\,dx\,ds\lesssim
\int_0^t\int|\nabla^2\,F|^2\,dx\,ds\\
&+\int_0^t\int (1+ |F|^{p-3})|\nabla\,F|^2|\nabla^2\,F|\,dx\,ds+\int_0^t\int (1+ |F|^{p-2})|\nabla\,F||\nabla^2\,v|\,dx\,ds\\
&+\int\left(|\nabla^2\,F_0|^2+|\nabla\,v_0|^2\right)\,dx=I_1+I_2+I_3+I_4.
\end{aligned}
\end{equation*}
The terms $I_1$, $I_3$ and $I_4$ can be treated in the same as before, while $I_2$ is treated as follows. First, since (A4$^\prime$) implies (A4) we can consider only the range $5<p\leq 6$. Then, by using Young's inequality, Lemma \ref{lem:gani} and the bound $\nabla F\in L^{\infty}(0,T;H^{1}(\T))$ we get that 
\begin{equation*}
\begin{aligned}
I_2&\lesssim \int_0^t\int |\nabla\,F|^2|\nabla^2\,F|\,dx\,ds+\int_0^t\int |F|^{p-4}|\nabla\,F|^4\,dx\,ds+\frac{1}{2}\int_0^t\int |F|^{p-2}|\nabla^2\,F|^2\,dx\,ds\\
&\lesssim \int_0^t\|\nabla^2F(s)\|_{2}^2\,ds+\int_0^t\left(\ln(e+\|\nabla^2F(s)\|_{2})\right)^{\frac{p-4}{2}}\|\nabla^2F(s)\|^2_{2}\,ds\\
&+\frac{1}{2}\int_0^t\int |F|^{p-2}|\nabla^2\,F|^2\,ds.
\end{aligned}
\end{equation*}
Then, by noticing that $5<p\leq 6$ implies that $\frac{p-4}{2}\in(0,1]$ we can conclude as the previous case. 
\end{proof}
%
%
%
%
%
%
%
%

\section{Sustained Oscillations}\label{sec:sust}

We next record some examples of solutions of the viscoelasticity system that exhibit sustained oscillations.

\subsection{A nonlinear problem with non-monotone stress-strain relation}
Consider the equations of viscoelasticity in one dimension, for the motion $y(t,x) :(0,T) \times [0,1] \to \R$,
\begin{equation}
\label{nonlinvis}
y_{tt} = \sigma (y_x)_x + y_{t x x}
\end{equation}
which upon setting $u = y_x$, $v= y_t$  is expessed as a system
\begin{equation}
\label{onedvisco}
\begin{aligned}
u_t &= v_x
\\
v_t &= \sigma(u)_x + v_{xx}.
\end{aligned}
\end{equation}
We denote the stress with $\sigma$ instead of $S$ to comply with the standard notation used in the one-dimensional case. The function $\sigma (u)$ is smooth and for the purposes of this section it is typically required
to be non-monotone.

Assume there exist two positive states $0 < a < b$ such that the values of the stress function $\sigma (u)$
satisfy 
\begin{equation}
\label{condnonm}
a + \sigma (t a) = b + \sigma (t b) \qquad \mbox{ $ \forall t \in [1,2]$} \, .
\end{equation}
Here,  the states $a,b$ are fixed and $t$ is thought of as a parameter. 

The condition \eqref{condnonm} restricts considerably the form of $\sigma(u)$, and we give an example
to show that it can be satisfied. Let $a, b >0$ fixed so that $0 < a < 2a < b < 2b$
and suppose the graph of $\sigma (u)$ is given and is strictly increasing for $u \in (b, 2b)$.
Then $(a, 2a) \cap (b, 2b) = \emptyset$ and  \eqref{condnonm} fully determines the graph of $\sigma (u)$ for $u \in (a, 2a)$
from the graph in $(b, 2b)$.
The emerging graph is increasing in $(a,2a)$ but the full graph will be non-monotone, see Figure  Fig. \ref{figgraph} where a
specific example is depicted.

\begin{figure}[htbp]
\begin{center}
\includegraphics[scale=0.3]{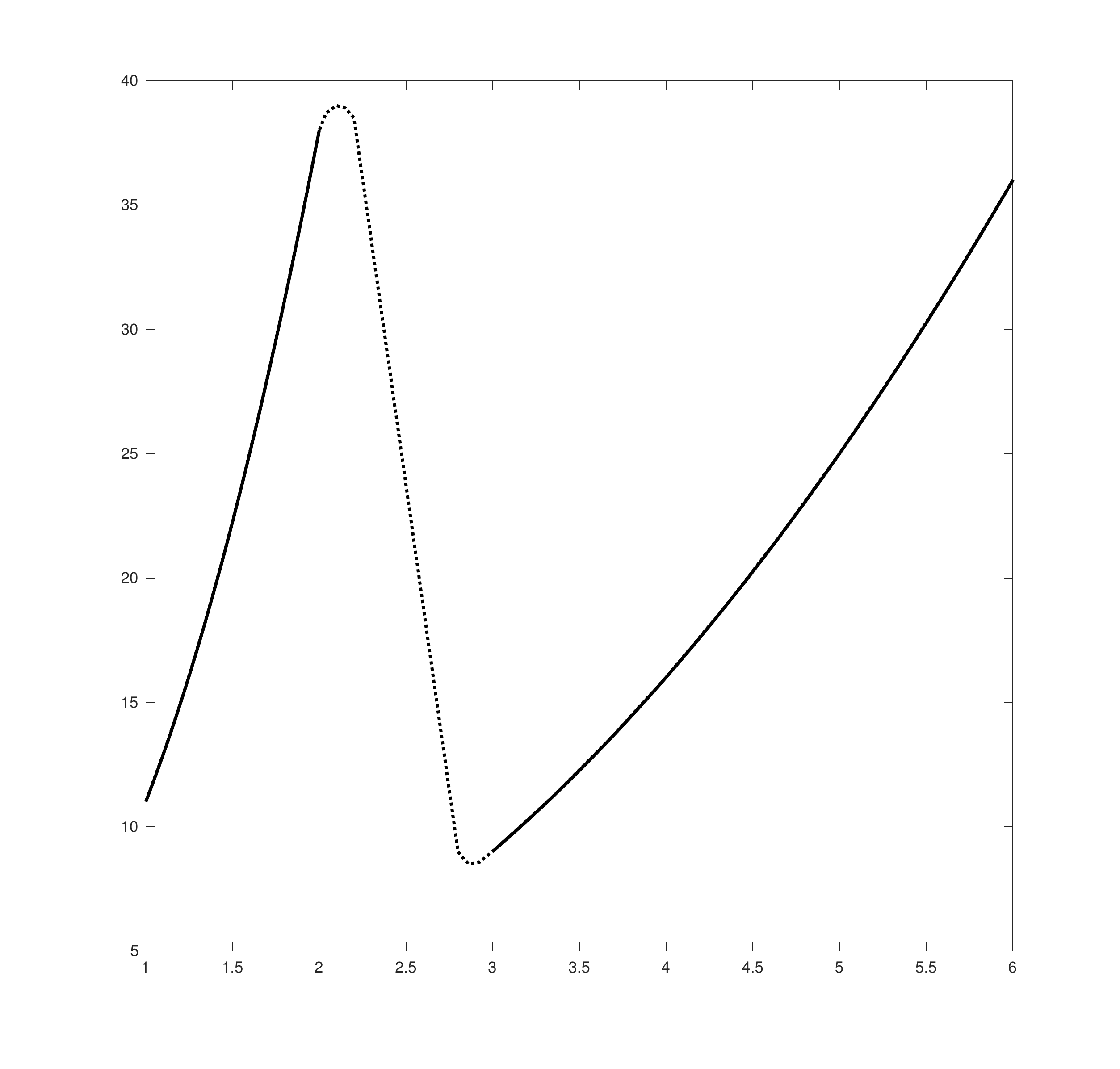}
\caption{The block parts of the graph satisfy \eqref{condnonm} for $1\le t\le2$; the dotted points interpolate between these parts.}
\label{figgraph}
\end{center}
\end{figure}



The example that is constructed below is based on two properties of \eqref{onedvisco}:
\begin{itemize}
\item the class  $u_s (t) = \kappa t$, $v_s (x) = \kappa x$ with $\kappa \in \R$
form a special class of universal -- for any $\sigma (u)$ -- uniform shear solutions.
\item as pointed out by {\sc D. Hoff} \cite{Hoff87} system \eqref{onedvisco} admits solutions 
continuous in $v$ but discontinuous in $u$ and $v_x$. Such solutions satisfy the Rankine-Hugoniot conditions
$$
\begin{aligned}
-s [u] &= [v]
\\
- s [v] &= [ \sigma (u) + v_x ]
\end{aligned}
$$
where $s$ is the shock speed and $[q] = q_+ - q_-$ the jump of the quantity $q$. As $v$ is continuous, 
the shocks are stationary $s=0$ and $[u] \ne 0$ has to satisfy
\begin{equation}\label{sec7RH}
s= 0 \, , \quad [ \sigma (u) + u_t ] = 0 \, .
\end{equation}
\end{itemize}

Next, we construct a family of solutions defined on $[1,2]\times \R$. To this end, fix states $a, b$ satisfying 
$0 < a < 2a < b < 2b$, 
suppose that \eqref{condnonm} is satisfied, and denote by $S(t)$ the common value
\begin{equation*}
S(t) := a + \sigma (t a) = b + \sigma (t b)  \qquad 1 \le t \le 2 \, .
\end{equation*}
For $0 < \theta < 1$, define the periodic function
$$
F(x) := \begin{cases}
 a,  &  x \in (k, k+\theta)
 \\
 b,  &  x \in (k + \theta , k + 1)
 \end{cases}
\quad \; k \in \mathbb{Z}
$$
and, based on $F$,  define
\begin{equation}
\label{soln1}
U(t,x) = t F(x) \, , \quad V_x (t,x) = F(x) \, , \quad t \in (1, 2)
\end{equation}
and  $V(t,x) \, , Y(t,x)$ are defined by setting
\begin{equation}
\label{soln2}
\begin{aligned}
V(t,x) &= \bar V (x)  := \int_0^x  F( y) dy 
=
\begin{cases}
k \bar V(1) + a  (x - k) ,  & x \in (k, k+\theta)
\\
k \bar V(1)  + a  \theta + b  (x - k - \theta),   & x \in (k+ \theta , k+1)
\end{cases}
\\
Y(t,x) &= \int_0^x U (t, y) dy = t \bar V (x) 
\end{aligned}
\end{equation}
where $ k \in \mathbb{Z}$, $t \in (1, 2)$ and $\bar V(1) = (a\theta + b(1-\theta))$.
The function $(U(t,x), V(t,x))$ is a weak solution of \eqref{onedvisco} on $ [1,2]\times\R$
satisfying the equations in a classical sense on the domains $(1,2)\times (k, k + \theta)$ and
$(1,2)\times(k + \theta , k+1)$ and due to \eqref{condnonm} it is a weak solution with discontinuities
at the interfaces $x = k$ and $x = k + \theta$ satisfying the Rankine-Hugoniot conditions \eqref{sec7RH}.

The functions $Y(t,x)$, $V(t,x)$, $U(t,x)$ are then rescaled and restricted to the interval $(1,2)\times (0,1)$
as follows
\begin{equation}
\label{exsoln3}
\begin{aligned}
y_n (t,x) &= \frac{1}{n} Y(t, nx) \, , \qquad  &&v_n (t,x) = \frac{1}{n} (\partial_t Y) (t, nx) = \frac{1}{n} V(t, nx)
\\
u_n (t,x)  &= U(t, nx) \, , \qquad  &&{v_n}_x (t,x) = V_x (t, nx).
\end{aligned}
\end{equation}
The equation \eqref{nonlinvis} is not invariant under the scaling $y_n (t,x) = \frac{1}{n} Y(t, nx)$.
However, since \eqref{soln1}, \eqref{soln2} is stationary, we easily check that \eqref{exsoln3}
are a class of exact weak solutions of \eqref{onedvisco} defined for $x \in (0,1)$, $t \in [1,2]$. 
Moreover, one easily computes the limits
$$
\begin{aligned}
v_n (t,x) &\to (a\theta + b(1-\theta)) x  \quad \mbox{ strongly in $L^2 \big ( (1,2)\times (0,1) \big )$ }
\\
u_n (t,x) &\rightharpoonup  (a\theta + b(1-\theta) ) t  \quad \mbox{ weakly-$\star$ in  $L^\infty \big ( (1,2)\times (0,1) \big )$ }
\\
&\sigma(u_n) + \partial_x v_n  = S(t)
\end{aligned}
$$
and 
$$
\sigma (u_n) \rightharpoonup  \theta \sigma ( at) + (1-\theta) \sigma(bt) \ne \sigma \big (  \theta a t + (1-\theta) b t) .
$$
The reader should note that the oscillations are induced by the oscillations in the initial data of $u_n (1,x)$.
\subsection{The linear problem}
Consider next the linear one-dimensional system
\begin{equation}
\label{linearvisc}
y_{tt} = \kappa y_{xx} + y_{t x x}  \qquad x \in (-\pi, \pi) \, , t > 0 
\end{equation}
where the motion $y(t,x)$ is now defined on the torus, $y : (0,T)\times \Tone \to \R$.
We investigate a class of oscillatory solutions of the form
\begin{equation*}
y_n = e^{ i n x} e^{ \lambda_n t} \, ;
\end{equation*}
these functions are periodic and will satisfy \eqref{linearvisc} provided $\lambda_n$ is a root of
\[
\lambda^2 + \lambda n^2 + \kappa n^2 = 0.
\]
The two roots are
\[
\lambda_{\pm} = n^2 \left ( -\tfrac{1}{2} \pm \tfrac{1}{2} \sqrt{1 - \tfrac{4 \kappa}{n^2} } \right ).
\]
Both roots are real and negative (for large $n$) and the smallest  in absolute value corresponds to the $+$ sign and
has the asymptotic behaviour
\[
\begin{aligned}
\lambda_{n +} &= n^2 \left ( -  \kappa \frac{1}{n^2} -  \kappa^2 \frac{2}{n^4} + \dots \right )
\\
&= - \kappa -  \kappa^2 \frac{2}{n^2} + \dots
\end{aligned}
\]
Consider now the rescaled solution $y_n = \tfrac{1}{n}  e^{ i n x} e^{- \lambda_n t}$ and observe that 
the associated $(u_n, v_n)$ have the behavior
$$
u_n (t,x) = i e^{ i n x} e^{- \lambda_{n+} t}  = i e^{ i n x - \kappa t} g_n (t) 
$$
where $g_n (t) \to 1$ and thus $u_n$ has persistent oscillations as $n \to \infty$.  Again such oscillations are induced from oscillatory initial data.
By contrast,
$$
v_n = - \frac{\lambda_{n+}}{n}  e^{ i n x} e^{- \lambda_{n+} t} \to 0 \qquad \mbox{as $n \to \infty$}.
$$
\appendix
	\section{Transfer of dissipation} \label{AppA}   
Consider the system \eqref{eq:main} with small viscosity $\eps > 0$, written
in coordinate form, 
\begin{equation}\label{eq:viscoapp}
\begin{aligned}
\partial_t\,v_i -\del_\alpha (S_{i \alpha} (F))- \eps \Delta\,v_i &=0\\
\partial_t\,F_{i \alpha}-\del_\alpha \,v_i &=0\\
\del_\alpha F_{i \beta} - \del_\beta F_{i \alpha} &= 0 \, ,
\end{aligned}
\end{equation}	
with periodic boundary conditions and initial data \eqref{eq:id}. We assume for the appendix that $W(F)$ is convex
and as usual $S_{i \alpha} = \frac{\del W}{\del F_{i \alpha}}$, and discuss the formal estimates 
available in the zero-viscosity limit.

First, the energy estimate reads
\begin{equation}
\label{idenergyapp}
\frac{d}{dt} \int \tfrac{1}{2} |v|^2 + W(F) \, dx + \int   \eps |\nabla v|^2 \, dx = 0
\end{equation}
and provides $L^p$ control for the family of solutions $(v^\eps, F^\eps)$ and the weak control $\eps  |\nabla v^\eps |^2 \in L^1_{t,x}$ familiar from
the theory of compensated compactness.

Using \eqref{eq:viscoapp} one obtains the identities
\[
\begin{aligned}
\big ( \del_\alpha S_{i \alpha} (F) \big ) (\del_\beta F_{i \beta}) -  (\del_\alpha v_i ) ( \del_\alpha v_i )
&=
(\del_\beta F_{i \beta} ) (\del_t v_i ) - (\del_\alpha v_i ) (\del_t F_{i\alpha}) - \eps (\del_\beta F_{i \beta}) \Delta v_i
\\
&= \del_t \big ( v_i \del_\beta F_{i \beta} \Big ) - \del_\alpha \big ( v_i \del_t F_{i \alpha} \big ) - \eps (\del_\beta F_{i \beta} ) \del_t ( \del_\alpha F_{i \alpha} \big ).
\end{aligned}
\]
Integrating over the torus gives
\begin{equation*}
\frac{d}{dt} \int \eps \frac{1}{2} | \dive F|^2  - v \cdot \dive F dx  + \int \big ( \del_\alpha S_{i \alpha} (F) \big ) (\del_\beta F_{i \beta}) -  (\del_\alpha v_i ) ( \del_\alpha v_i ) dx
= 0 .
\end{equation*}
Next, a use of integration by parts gives
$$
\begin{aligned}
I &= \int \big ( \del_\alpha S_{i \alpha} (F) \big ) (\del_\beta F_{i \beta}) -  |\nabla v|^2 dx
\\
&= \int \frac{ \del^2 W}{\del F_{i \alpha} \del F_{k \gamma}} \del_\beta F_{k \gamma} \del_\alpha F_{i \beta} - |\nabla v|^2 dx
\\
&= \int \frac{ \del^2 W}{\del F_{i \alpha} \del F_{k \gamma}} \del_\beta F_{k \gamma} \del_\beta F_{i \alpha} - |\nabla v|^2 dx
\\
&= \int D^2 W : (\nabla F, \nabla F) - |\nabla v|^2 dx
\end{aligned}
$$
and implies
\begin{equation}
\label{idvisccomp}
\frac{d}{dt} \int \eps^2 \frac{1}{2} | \dive F|^2  - \eps v \cdot \dive F \, dx  + \int  \eps ( D^2 W : (\nabla F, \nabla F) - |\nabla v|^2 ) \, dx = 0.
\end{equation}
Combining \eqref{idenergyapp} with \eqref{idvisccomp} we arrive at
\begin{equation}
\label{unic}
\frac{d}{dt} \int  \tfrac{1}{2} |  v - \tfrac{\eps}{2} \dive F|^2 + \tfrac{\eps^2}{4} |\dive F |^2+ W(F) \, dx 
+ \frac{\eps}{2} \int   D^2 W : (\nabla F, \nabla F) +  |\nabla v|^2  \, dx = 0.
\end{equation}
Identity \eqref{unic} imples that, under conditions of uniform convexity for $W(F)$,
$$
D^2 W \ge c \; \mathbb{I} \, , \quad \mbox{for some $c > 0$},
$$ 
in addition to the uniform bounds $\eps |\nabla v^\eps|^2 \in_b L^1_{t,x}$ one also obtains control $\eps  |\nabla F^\eps|^2  \in_b L^1_{t,x}$. 
The estimate \eqref{unic} extends to several space dimensions an observation of DiPerna \cite{DiPerna83} in connection to the problem of zero-viscosity limits in one space dimension.
A similar property holds for relaxation approximations of the nonlinear elasticity system in one-space dimension \cite{Tzavaras99}.

\begin{remark}
Identity \eqref{unic} in conjunction with the energy \eqref{idenergyapp} provides the main a-priori estimates for system \eqref{eq:main}.
If for example one assumes hypothesis \eqref{eq:ab} and use $\eps =1$,  Lemma \ref{lem:modAB} and the identity
then we conclude
\begin{equation}
\label{unicab}
\frac{d}{dt} \int   |  v - \tfrac{1}{2} \dive F|^2 + \tfrac{1}{2} |\nabla F |^2+  2 W(F) \, dx 
+  \int  \Big (  D^2 \tilde W : (\nabla F, \nabla F) +  |\nabla v|^2 \Big ) dx \le K \int |\nabla F|^2 dx.
\end{equation}
Using Gr\"onwall's lemma, \eqref{unicab} yields control of $\int |\nabla F|^2 dx$.
\end{remark}

\section{Diffusion-dispersion approximations of elasticity} \label{AppB} 
Diffusion-dispersion approximations arise  naturally in studies of phase transitions in elasticity.
One way to introduce such systems is to consider a nonlinear evolution 
$y(t,x)$ generated by an energy functional
$$
\frac{\del^2 y}{\del t^2} = - \frac{\delta \cE}{\delta F}[y] \, ,
$$
where $\cE [y]$ is a functional on the motion $y$. A simple example is provided by strain
gradient theories of the form
$$
\cE [y] = \int W(\nabla y) + \delta A \frac{1}{2} |\Delta y |^2 dx \, .
$$
Here, $W(F)$ is assumed nonconvex to allow models with phase transitions and the term with the
higher order gradient is motivated by the Korteweg theory.
It leads to the second order nonlinear partial differential equation
\begin{equation}
\label{eq:capil}
\frac{\del^2 y}{\del t^2} = \dive \Big ( \frac{\del W}{\del F} (\nabla y) - \delta A  \nabla \Delta y \Big ) \, .
\end{equation}
It is known that for strain energies that are non-convex, viscosity is not sufficient to select the admissible shocks and, 
motivated by the Korteweg theory, Slemrod \cite{Slemrod83} and Truskinovsky \cite{Truskinovskii82} proposed to include the effects of capillarity.

Adding viscosity to \eqref{eq:capil} produces a diffusive-dispersive approximation of the elasticity equations,
in the form of the system
\begin{equation}
\label{difdis}
\begin{aligned}
\partial_t\,v_i - \del_\alpha S_{i \alpha} (F) &= \eps \Delta\,v_i  -
 \delta A \del_\alpha \Delta F_{i \alpha}
\\
\partial_t\,F_{i \alpha}-\del_\alpha \,v_i &=0
\\
\del_\alpha F_{i \beta} - \del_\beta F_{i \alpha} &= 0 \, 
\end{aligned}
\end{equation}	
with $W(F)$ is nonconvex.
Here $\eps > 0$,  $\delta > 0$ are positive parameters while $A$ is a numerical constant.
The involution \eqref{difdis}$_3$ is a constraint induced on solutions by the initial data.
In one space dimension in the limit as $\eps, \delta \to 0$ diffusion dominates dispersion in (at least) the range $\delta = O(\eps^2)$.
A  clever transformation of variables discovered by Slemrod \cite{Slemrod89} indicates that  \eqref{difdis} in 1-$d$ can be transformed to
a viscosity approximation for a system of conservation laws. 

A generalization of this observation to multi-$d$ is provided here.
Let $\kappa$ be a parameter (to be selected) and write \eqref{difdis}$_1$, \eqref{difdis}$_2$, respectively,  as follows:
\[
\begin{aligned}
\del_t ( v_i - \kappa \del_\gamma F_{i \gamma}) - \del_\alpha S_{i \alpha} (F) &= (\eps - \kappa) \Delta ( v_i - \kappa \del_\alpha F_{i \alpha} )
+ [ (\eps - \kappa ) \kappa - \delta A ] \del_\alpha  \Delta F_{i \alpha}
\\
\del_t F_{i \alpha} - \del_\alpha \big ( v_i - \kappa \del_\gamma F_{i \gamma} \big ) &= \kappa \Delta F_{i \alpha}.
\end{aligned}
\]
Observe that if  $\kappa$ is selected by
\begin{equation}
\label{condquad}
\begin{aligned}
\kappa^2 - \eps \kappa + \delta A = 0
\\
0 < \kappa < \eps 
\end{aligned}
\end{equation}
then  \eqref{difdis} reduces to the hyperbolic parabolic system
\begin{equation}
\label{difdisred}
\begin{aligned}
\del_t w - \dive S (F) &= (\eps - \kappa) \Delta w
\\
\del_t F -  \nabla w  &= \kappa \Delta F
\\
\curl F &= 0
\end{aligned}
\end{equation}
which describes the evolution of the function $(w, F)$ with $w = v - \kappa \dive F$. Again $\curl F = 0$ is an involution propagated from the initial data
via the equation \eqref{difdisred}$_2$.

One easily checks the solvability of \eqref{condquad}. The roots of the quadratic are
$$
\kappa_{\pm} = \frac{\eps}{2} \pm \frac{\eps}{2} \sqrt{1 - \frac{4 \delta A}{\eps^2} }.
$$
We deduce that if $\delta = \eps^\rho$, $\rho > 2$,  then for any $A$ we can select $\kappa$ that fulfills the condition $0 < \kappa < \eps$;
in the borderline case $\delta = \eps^2$  the parameter $A$ is restricted to be $0 < A \le \frac{1}{4}$. 
An interesting special case occurs for $\delta = \eps^2$, $A = \frac{1}{4}$
and leads to a system with identity viscosity matrix
$$
\begin{aligned}
\del_t w - \dive S (F) &= \tfrac{\eps}{2}  \Delta w
\\
\del_t F -  \nabla w  &= \tfrac{\eps}{2} \Delta F \, .
\end{aligned}
$$

\section{A discussion on the assumptions on the stored energy}\label{AppC}
We briefly comment on the set of assumptions (A1)-(A4) on $W$, which we rewrite for the reader's convenience. For $p\geq 2$ we assume $W$ satisfies 
\begin{itemize}
\item[(A1)] $W\in C^{2}(\R^{d\times d};\R)$.
\item[(A2)] There exists $c>0$ such that 
\begin{equation*}
c(|F|^p-1)\leq W(F).
\end{equation*}
\item[(A3)] There exists $C>0$ such that 
\begin{equation*}
\begin{aligned}
&|W(F)|\leq C(1+|F|^p);\,&|DW(F)|\leq C(1+|F|^{p-1}).\\
\end{aligned}	
\end{equation*}
\item[(A4)] There exists constant $K>0$ such that 
\begin{equation*}
(DW(F_1)-DW(F_2), F_1-F_2)\geq -K|F_1-F_2|^2.
\end{equation*}
\end{itemize}
We start by remarking that the growth condition on $DW$ in (A3) is redundant if we assume (A4). Indeed, by Lemma \ref{lem:modAB}, we have that $\tilde{W}$ is convex with a $p$-growth, because of the growth of $W$. Then, it is well-known that $D\tilde{W}$ must have a $p-1$-growth,  see for example \cite[Proposition 2.32]{D}, and therefore $DW$ has a $p-1$-growth as well.\par
A technical, yet necessary assumption in the analysis of the uniqueness problem is  \eqref{eq:growthsecond}, that is the following condition on the second derivative on $W$: there exists $C>0$ such that 
\begin{equation}\label{eq:growthsecondapp}
\begin{aligned}
&|D^2W(F)|\leq C(1+|F|^{s}),\mbox{ with }s\geq p-2.
\end{aligned}
\end{equation}
The following example shows that \eqref{eq:growthsecondapp} is not a consequence of (A1)-(A4). 
\begin{example}
Let $n=1$ and $\phi:\R\mapsto \R$ be the function $\phi(x):=e^{\frac{1}{1-|x|^2}}\chi_{[-1,1]}(x)$. Let $k\in\N$ with $k\geq 2$ and define 
\begin{equation*}
\phi_{k}(x):=e^{k}\phi(e^{3k}(x-k)).
\end{equation*}
Note that $supp\,\phi_k=[k-e^{-3k}, k+e^{-3k}]$. Let $g\in C(\R^{+};\R^{+})$ given by 
\begin{equation*}
g(x):=1+\sum_{k=2}^{\infty}\phi_{k}(x),\, x>0.
\end{equation*}
Define $f:\R\to\R$ first by solving for $x>0$ 
\begin{equation}\label{eq:odeC2}
 \begin{aligned}
 &f''(x)=g(x),\quad&f'(0)=0,\quad &f(0)=0,
 \end{aligned}
\end{equation}
and then by extending the resulting $f$ in an even way for negative $x$'s. Then, by construction $f''$ saturates the exponential growth. Moreover, $f\in C^2(\R;\R)$ and it is convex. Therefore it satisfies (A4). It remains to check that the growth and coercivity conditions are satisfied. It is enough to consider the case $x>0$.  Integrating \eqref{eq:odeC2} we have that , since $\phi_k$ is positive for any $k$, 
\begin{equation*}
\begin{aligned}
x\leq f'(x) &\leq x+\sum_{k=2}^{\infty}\int_{k-e^{-3k}}^{k+e^{-3k}}e^{k}\phi(e^{3k}(y-k))\,dy
&\leq x+\sum_{k=2}^{\infty}e^{-2k}\int_{-1}^1\phi(z)\,dz
&\leq C(x+1).
\end{aligned}
\end{equation*}
Then, after integrating again we also get that $\frac{x^2}{2}\leq f(x)\leq C(x^{2}+1)$ and therefore $f$ satisfies also (A2) and (A3). 
\end{example}
We also note that, because of the coercivity assumption (A2), the lower bound on $s$ in \eqref{eq:growthsecondapp} is necessary. 
\begin{lemma}\label{lem:app1}
Let $k\in\N$, $p>k$ and $f\in C^{k}(\R^n;\R)$ $p$-coercive. Then, if there exists $C>0$ and $s\geq 0$ such that $|D^k f(x)|\leq C(1+|x|^s)$ for any $x\in\R^n$, it must hold that $s\geq p-k$.  
\end{lemma}
\begin{proof}
Without loss of generality we can assume that $n=1$, $f$ positive and $x>0$. Moreover, it is enough to prove the lemma for $k=1$. For general $k$ it will follow by induction. Let $p>1$, $f\in C^{1}(\R;\R)$, $C>0$, and $s>0$ such that for any $x>0$
\begin{equation*}
\begin{aligned}
&C(x^p-1)\leq f(x);\\
&|f'(x)|\leq C(1+x^s). 
\end{aligned}
\end{equation*}
Assume there exists $\e>0$ such that $0\leq s=p-1-\e$. Then, for some $C>0$  
\begin{equation*}
C(x^p-1)\leq f(x)\leq \int_0^x|f'(y)|\,dy+C\leq C(1+x+x^{s+1}).
\end{equation*}
Therefore, for $C>0$ we have that for any $x>0$ 
\begin{equation*}
x^{p}\leq C(1+x+x^{p-\e})
\end{equation*}
which is a contradiction since $p>1$. 
\end{proof}

\end{document}